\documentclass[12pt, reqno]{amsart}

\usepackage{
amsmath,
amssymb,
amsthm,
mathrsfs,
amsfonts,
ytableau,
enumitem,
comment,
upgreek,
soul,
yhmath,
mathtools,
shuffle,
kotex,
array,
caption,
hhline %표에 double line할 수 있게 해줌.
}

\usepackage[colorlinks=true, pdfstartview=FitV, linkcolor=blue,citecolor=blue, urlcolor=blue]{hyperref}

\usepackage{tikz,tikz-cd}

\usepackage[colorinlistoftodos, textwidth = 2.3cm]{todonotes}

\usepackage{bbm}
\usepackage{dsfont}
\usepackage{bm}

\ytableausetup{mathmode, boxsize=1.2em}

\title[Categorifications of ${\textsf {QSym}}$ using supercharacter theories] 
{Categorifications of ${\textsf {QSym}}$ using supercharacter theories and a new basis for ${\textsf {NSym}}_{\mathbb{C}(q,t)}$}

\author[W.-S. Jung]{Woo-Seok Jung}
\address{Department of Mathematics, Sogang University, Seoul 04107, Republic of Korea}
\email{jungws@sogang.ac.kr}

\author[Y.-T. Oh]{Young-Tak Oh}
\address{Department of Mathematics, Sogang University, Seoul 04107, Republic of Korea}
\email{ytoh@sogang.ac.kr}

\thanks{All authors were supported by the National Research Foundation of Korea (NRF) Grant funded by the Korean Government (NRF-2020R1F1A1A01071055).}

\keywords{Quasisymmetric function, Noncommutative symmetric function, Hopf algebra, Supercharacter, Categorification }
\subjclass[2020]{05E05, 05E10, 16T05, 20C05}
\date{\today}

\newtheorem{theorem}{Theorem}[section]
\newtheorem{proposition}[theorem]{Proposition}
\newtheorem{lemma}[theorem]{Lemma}
\newtheorem{corollary}[theorem]{Corollary}

\newtheorem*{claim*}{Claim}

\theoremstyle{definition}

\newtheorem{example}[theorem]{Example}
\newtheorem{definition}[theorem]{Definition}
\newtheorem{remark}[theorem]{Remark}

\numberwithin{equation}{section} \numberwithin{figure}{section}
\numberwithin{table}{section}

\newcommand{\nc}{\newcommand}

\nc{\SG}{\mathfrak{S}}
\nc{\frakR}{\mathfrak{R}}
\nc{\frakL}{\mathfrak{L}}
\nc{\PCT}{\mathrm{PCT}}
\nc{\SPCT}{\mathrm{SPCT}}
\nc{\RT}{\mathrm{RT}}
\nc{\SRT}{\mathrm{SRT}}
\nc{\RCT}{\mathrm{RCT}}
\nc{\SRCT}{\mathrm{SRCT}}
\nc{\SYRT}{\mathrm{SYRT}}
\nc{\SYCT}{\mathrm{SYCT}}
\nc{\SPYCT}{\mathrm{SPYCT}}
\nc{\tst}{\mathtt{st}}
\nc{\Span}{\mathrm{span}}
\nc{\comp}{\mathrm{comp}}
\nc{\rmst}{\mathrm{st}}
\nc{\std}{\mathrm{std}}
\nc{\Des}{\mathrm{Des}}
\nc{\set}{\mathrm{set}}
\nc{\wt}{\mathrm{wt}}
\nc{\cf}{\textsf{cf}}
\nc{\scf}{\textsf{scf}}
\nc{\ch}{\mathrm{ch}}
\nc{\cl}{\mathrm{cl}}
\nc{\id}{\mathrm{id}}
\nc{\sh}{\mathrm{sh}}
\nc{\Cop}{\mathrm{Cop}}
\nc{\bfS}{\mathbf{S}}
\nc{\bfm}{\mathbf{m}}
\nc{\hbfS}{\widehat{\mathbf{S}}}
\nc{\bfF}{\mathbf{F}}
\nc{\calB}{\mathcal{B}}
\nc{\calS}{\mathcal{S}}
\nc{\hcalS}{\widehat{\mathcal{S}}}
\nc{\alphamax}{\alpha_{\rm max}}
\nc{\brho}{\overline{\rho}}
\nc{\bphi}{\overline{\phi}}
\nc{\calV}{\mathcal{V}}
\nc{\calR}{\mathcal{R}}
\nc{\sfR}{\mathsf{R}}
\nc{\calG}{\mathcal{G}}
\nc{\tal}{\lambda(\alpha)}
\nc{\tbe}{\widetilde{\beta}}
\nc{\opi}{\overline{\pi}}
\nc{\calP}{\mathcal{P}}
\nc{\rmtop}{\mathrm{top}}
\nc{\rad}{\mathrm{rad}}
\nc{\bfP}{\mathbf{P}}
\nc{\SET}{\mathrm{SET}}
\nc{\SIT}{\mathrm{SIT}}
\nc{\rev}{\mathrm{r}}
\nc{\Th}{\theta}
\nc{\mPhi}{\Phi}
\nc{\mphi}{\phi}
\nc{\mPsi}{\Psi}
\nc{\hmPsi}{\widehat{\Psi}}
\nc{\mpsi}{\psi}
\nc{\mGam}{\Gamma}
\nc{\tcd}{\mathtt{cd}}
\nc{\trd}{\mathtt{rd}}
\nc{\trcd}{\mathtt{rcd}}
\nc{\rmr}{\mathrm{r}}
\nc{\rmc}{\mathrm{c}}
\nc{\rmt}{\mathrm{t}}
\nc{\bubact}{\,\scalebox{0.6}{$\bullet$}\,}
\nc{\hbubact}{\,\scalebox{0.6}{$\widehat{\bullet}$}\,}
\nc{\col}{\rm col}
\nc{\row}{\rm row}
\nc{\calE}{\mathcal{E}}
\nc{\calT}{\mathscr{T}}
\nc{\sfT}{\mathsf{T}}
\nc{\calEsa}{\mathcal{E}^\sigma(\alpha)}
\nc{\tauC}{\tau_{\scalebox{0.5}{$C$}}}
\nc{\sytabC}{\sytab_{\scalebox{0.5}{$C$}}}
\nc{\bbfP}{\overline{\bfP}}
\nc{\pr}{\mathbf{pr}}
\nc{\Ups}{\Upsilon}
\nc{\pact}{\diamond}
\nc{\tauE}{\tau_{\scalebox{0.5}{$E$}}}
\nc{\tauF}{\tau_{\scalebox{0.5}{$F$}}}
\nc{\tauG}{\tau_{\scalebox{0.5}{$G$}}}
\nc{\rtE}{T_{\scalebox{0.5}{$E$}}}
\nc{\rtF}{T_{\scalebox{0.5}{$F$}}}
\nc{\rtG}{T_{\scalebox{0.5}{$G$}}}
\nc{\oPaE}{\overline{\Phi}_{\alpha_E}}
\nc{\oPaF}{\overline{\Phi}_{\alpha_F}}
\nc{\oPaG}{\overline{\Phi}_{\alpha_G}}
\nc{\tab}{\tau}
\nc{\sytab}{\widehat{\tau}}
\nc{\hatE}{\widehat{E}}
\nc{\hati}{\hat{i}}
\nc{\hcalE}{\widehat{\calE}}
\nc{\hatC}{\widehat{C}}
\nc{\bal}{{\boldsymbol{\upalpha}}}
\nc{\bbe}{{\boldsymbol{\upbeta}}}

\nc{\bgam}{{\boldsymbol{\upgamma}}}
\nc{\bdel}{{\boldsymbol{\updelta}}}
\nc{\weakcon}{\odot}
\nc{\basisI}{I}
\nc{\ldalpha}{\lambda(\alpha)}
\nc{\SRIT}{\mathrm{SRIT}}
\nc{\re}{\mathrm{rev}}
\nc{\otau}{\overline{\tau}}
\nc{\rtop}{{\rm top}}
\nc{\sfc}{\mathsf{c}}
\nc{\sfr}{\mathsf{r}}
\nc{\tH}{\mathtt{H}}
\nc{\tV}{\mathtt{V}}
\nc{\rpi}{\mathring{\pi}}
\nc{\cpi}{\check{\pi}}
\nc{\frakm}{\mathfrak{m}}
\nc{\fke}{\mathfrak{e}}
\nc{\Hom}{\mathrm{Hom}}
\nc{\module}{\mathrm{mod} \, }

\nc{\SPCTsa}{\SPCT^\sigma(\alpha)}
\nc{\bfSsa}{\bfS_\alpha^\sigma}
\nc{\bfSsaC}{{\bfS}^\sigma_{\alpha,C}}
\nc{\hbfSsa}{\widehat{\bfS}_\alpha^\sigma}
\nc{\upineq}{\rotatebox{90}{$<$}}
\nc{\downineq}{\rotatebox{270}{$<$}}
\nc{\diagineq}{\rotatebox{135}{$<$}}
\nc{\frakB}{\mathfrak{B}}
\nc{\hxi}{\widehat{\xi}}
\nc{\hxidwJ}{\hxi_{\scalebox{0.55}{$J$}}}
\nc{\hxiupJ}{\hxi^{\scalebox{0.55}{$J$}}}
\nc{\scrS}{\mathscr{S}}
\nc{\bfT}{\mathbf{T}}

\nc{\ra}{\rightarrow}

\nc{\matr}[2]{\left( \hspace{-1ex} \begin{array}{c} #1 \\ #2 \end{array} \hspace{-1ex} \right)}

\newcommand{\Sym}{{\textsf {Sym}}}
\newcommand{\QSym}{{\textsf {QSym}}}
\newcommand{\NSym}{{\textsf {NSym}}}
\newcommand{\FQSym}{{\textsf {FQSym}}}
\newcommand{\NCSym}{\textsf{NCSym}}
\newcommand{\Comp}{{\textsf {Comp}}}
\newcommand{\Par}{{\textsf {Par}}}
\newcommand{\Irr}{{\rm Irr}}

\definecolor{wsgreen}{rgb}{0,0.5,0}

\newcommand{\I}{\mathfrak{S}}

\nc{\DIRT}{\mathrm{DIRT}}
\nc{\hpi}{\pi}
\nc{\frakI}{\mathfrak{I}}
\nc{\hfrakI}{\widehat{\mathfrak{I}}}
\nc{\orho}{\overline{\rho}}
\nc{\autotheta}{\uptheta}
\nc{\autophi}{\upphi}
\nc{\autochi}{\upchi}
\nc{\autoomega}{\upomega}
\nc{\hIM}{\widehat{\frakB}}
\nc{\bfpi}{\boldsymbol{\uppi}}
\nc{\bfopi}{\overline{\boldsymbol{\uppi}}}
\nc{\ofrakB}{\overline{\frakB}}
\nc{\rmw}{\mathrm{w}}
\nc{\ostar}{\;\overline{*}\;}
\nc{\rank}{\mathrm{rank}}
\nc{\fkp}{\mathfrak{p}}
\nc{\bfR}{\mathbf{R}}

\nc{\upsig}{{\boldsymbol{\upsigma}}}
\nc{\bfSsaE}{{\bfS}^\upsig_{\alpha,E}}

\nc{\hfkp}{\widehat{\mathfrak{p}}}

\nc{\hautophi}{{\widehat{\autophi}}}
\nc{\hautotheta}{{\widehat{\autotheta}}}
\nc{\hautoomega}{{\widehat{\autoomega}}}
\nc{\rmperm}{\mathrm{perm}}

\nc{\bfsigJ}{\boldsymbol{\sigma}_{\scalebox{0.55}{$J$}}}
\nc{\bfrhoJ}{\boldsymbol{\rho}^{\scalebox{0.55}{$J$}}}

\nc{\pistar}[1]{\pi_{#1}^*}
\nc{\wfkp}{\widetilde{\mathfrak{p}}}
\nc{\bfpsi}{\boldsymbol{\uppsi}}

\nc{\yt}[1]{\todo[size=\tiny,color=blue!10]{#1 \\ \hfill --- Young-Tak}}
\nc{\YT}[1]{\todo[size=\tiny,inline,color=blue!10]{#1
		\\ \hfill --- Young-Tak}}

\nc{\ws}[1]{\todo[size=\tiny,color=green!10]{#1 \\ \hfill ---  Woo-Seok}}
\nc{\WS}[1]{\todo[size=\tiny,inline,color=green!10]{#1
		\\ \hfill --- Woo-Seok}}

\definecolor{purple}{rgb}{0.44, 0.0, 1.0}

\newenvironment{red}{\relax\color{red}}{\hspace*{.5ex}\relax}
\newenvironment{blue}{\relax\color{blue}}{\hspace*{.5ex}\relax}
\newenvironment{green}{\relax\color{wsgreen}}{\hspace*{.5ex}\relax}
\newenvironment{magenta}{\relax\color{magenta}}{\hspace*{.5ex}\relax}
\newenvironment{purple}{\relax\color{purple}}{\hspace*{.5ex}\relax}

\nc{\ber}{\begin{red}}
\nc{\er}{\end{red}}
\nc{\beb}{\begin{blue}}
\nc{\eb}{\end{blue}}
\nc{\bema}{\begin{magenta}}
\nc{\ema}{\end{magenta}}
\nc{\begr}{\begin{green}}
\nc{\egr}{\end{green}}

\nc{\bepu}{\begin{purple}}
\nc{\epu}{\end{purple}}
\nc{\lb}{\pmb{\left[\vphantom{\frac{1}{2}}\right.}}
\nc{\rb}{\pmb{\left.\vphantom{\frac{1}{2}}\right]}}
\nc{\slb}{\pmb{[\vphantom{\frac{1}{2}}}}
\nc{\srb}{\pmb{\vphantom{\frac{1}{2}}]}}
\begin{document}

\maketitle

%%%%%%%%%%%%%%%%%%%%%%%%%%%%%%%%%%%%%%%%%%%%%%%%%%%%%%%%%%%%%%%%%%%%

\begin{abstract}
Let us fix a positive integer $\nu>1$. For each positive integer $n>1$, we consider a normal supercharacter theory $\mathcal{S}_n$ of $G_n$, where $G_n$ is the direct-product of $n-1$ copies of the cyclic group of order $\nu$. Then we endow $\bigoplus_{n \ge 0} \textsf{scf}(\mathcal{S}_n)$, the direct-product of supercharacter function spaces, with the Hopf algebra structure that is isomorphic to the Hopf algebra $\textsf{QSym}$ of quasisymmetric functions. Furthermore, we compute the structure constants of the Hopf algebra thus obtained for the basis consisting of superclass identifier functions. Using our categorifications, we study a new basis for the Hopf algebra $\textsf{NSym}_{\mathbb{C}(q,t)}$ of noncommutative symmetric functions over the rational function field $\mathbb{C}(q,t)$ in commuting variables $q$ and $t$, with an emphasis on the structure constants of $\textsf{NSym}_{\mathbb{C}(q,t)}$ for this basis. Some interesting applications are also obtained via the specializations of $q$ and $t$.
\end{abstract}

\tableofcontents

\section{Introduction}

Categorifying an algebraic object allows us to look at the hidden side of the object. 
Through this process, we can derive information that is difficult to see from the object itself.
Perhaps the most foundational example is the categorification of the Hopf algebra $\Sym$ of symmetric functions via the tower of the symmetric group algebras $\mathbb C \SG_n$, where  $n$ ranges over the set of nonnegative integers. 
More precisely, equipped with induction product and restriction, the Grothendieck group associated with $\bigoplus_{n\ge 0} \mathbb C \SG_n$ turns out to be isomorphic to $\Sym$ as Hopf algebras (\cite{G77, Z81}).
In fact, this categorification has a tremendous influence on the theory of symmetric functions (see~\cite{M98} and~\cite{FH13}).

In this paper, we focus on the Hopf algebra $\QSym$ of quasisymmetric functions.
In~\cite{DKLT96}, Duchamp, Krob, Leclerc, and Thibon categorified $\QSym$ via the tower of the $0$-Hecke algebras $H_n(0)$.
Similar to the above, equipped with induction product and restriction, the Grothendieck group associated with $\bigoplus_{n\ge 0} H_n(0)$ turns out to be isomorphic to $\QSym$ as Hopf algebras.
However, since the $0$-Hecke algebras are not group algebras, 
one can naturally ask whether $\QSym$ can be categorified via the tower of certain group algebras as in $\Sym$.
Here we give a positive answer to this question.
In more detail, we successfully categorify $\QSym$ via supercharacter theories of finite abelian groups. 
The results thus obtained will be used in studying a new basis for the Hopf algebra $\NSym$ of noncommutative symmetric functions.

Let us briefly review supercharacter theories of finite groups.
Classifying the conjugacy classes and the irreducible characters of the unipotent upper triangular matrix group $UT_n(q)$ over a finite field $\mathbb{F}_q$, where $n$ ranges over the set of nonnegative integers, has been regarded as a difficult problem.
Supercharacter theory was initiated by André(\cite{A95}) and Yan(\cite{Y01}) as a tractable method for studying the representation theory of $UT_n(q)$. 
Inspired by André’s and Yan’s work, Diaconis and Isaacs~\cite{DI08} extended the concept of supercharacter theories to arbitrary finite groups in an axiomatic method.

Given a finite group $G$, a \emph{supercharacter theory} $\mathcal{S}$ of $G$ is defined as a pair of a set partition of $\{\text{conjugacy classes of $G$}\}$ and a set partition of $\{\text{irreducible characters of $G$}\}$ that satisfies appropriate axioms (Definition~\ref{def: supercharacter theory}).
The following table illustrates key analogies between the character theory and a supercharacter theory well 
(for more information, see Subsection~\ref{subsection: Supercharacter theories of a finite group}).
{\small
\begin{table}[h]
\begin{tabular}{|c|c|}
\hline
 \textbf{The character theory of $G$ }   &  \textbf{ A supercharacter theory of $G$}  \\ \hline
conjugacy classes                 & superclasses                              \\ \hline
irreducible characters            & supercharacters                          \\ \hline

class identifier functions        & superclass identifier functions          \\ \hline

class functions             & supercharacter functions                                       \\ \hline
\end{tabular}
\end{table}}

Supercharacter theories have been used extensively in the problems of categorifying combinatorial Hopf algebras.
In 2013, using André's supercharacter theory of $UT_n(q)$, Aguiar et al.~\cite{28people12} succeeded in categorifying the Hopf algebra $\NCSym$ of symmetric functions in noncommuting variables.
Very recently, using normal supercharacter theories, Aliniaeifard and Thiem succeeded in categorifying the Hopf algebra $\FQSym$ of free quasisymmetric functions in~\cite{AT21} and the Hopf algebra $\NSym$ of noncommutative symmetric functions in~\cite{AT21-Nsym}.
Since $\FQSym$, $\QSym$, $\NSym$, and $\Sym$ are the four most classical and important combinatorial Hope algebras, it would be very natural to deal with $\QSym$ as a follow-up study of the above works.

The first objective of the present paper is to categorify 
$\QSym$ using normal supercharacter theories of finite abelian groups.
Our setup is as follows.
Let $n$ and $\nu$ be any positive integers $>1$, $C_\nu$ the additive cyclic group of order $\nu$, and $G$ the direct product of $n-1$ copies of $C_\nu$.  
Then we let
\[\mathcal{N}_n(\nu):= \{ Q_I(\nu) \mid I \subseteq [n-1] \},\]
where $Q_I(\nu)$ is the subgroup of $G$ whose $j$th component is $C_\nu$ if $j\in I$ and $\{0\}$ otherwise.
According to Aliniaeifard's result~\cite[Theorem 3.4]{A17},
it gives rise to a normal supercharacter theory $\mathcal{S}(\mathcal{N}_n(\nu))$. 
Denote by 
$\{ \chi^I(\nu) \mid I \subseteq S \}$ and $
\{ \kappa_I(\nu) \mid I \subseteq S \}$
the set of supercharacters and the set of superclass identifier functions of $\mathcal{S}(\mathcal{N}_n(\nu))$,
respectively. For full information, see Subsection~\ref{subsection: Normal supercharacter theories of finite abelian groups}.
It can be easily checked that the supercharacter function space of $\mathcal{S}(\mathcal{N}_n(\nu))$ has the same dimension as $\QSym_n$, the $n$th homogeneous component of $\QSym$.

We elaborately define a product ${\bf m}$ and a coproduct $\blacktriangle$ on the $\mathbb C$-vector space 
\[\bigoplus_{n \ge 0} \scf(\mathcal{S}(\mathcal{N}_n(\nu)))\] 
as compositions of certain linear operators. 
Here $\scf(\mathcal{S}(\mathcal{N}_n(\nu)))$ is the supercharacter function
of $\mathcal{N}_n(\nu)$ and $\scf(\mathcal{S}(\mathcal{N}_n(\nu)))\cong \mathbb C$ for $n=0,1$.
Equipped with these operations, it is proven that 
$\bigoplus_{n \ge 0} \scf(\mathcal{N}_n(\nu))$
has a Hopf algebra structure and the characteristic map 
\begin{align*}
    \ch_\nu : \bigoplus_{n \ge 0} \scf(\mathcal{N}_n(\nu)) \to \QSym, \quad  \dot\chi^I(\nu)\mapsto L_{\comp(I)} \quad (I \subset [n-1])
\end{align*}
is an isomorphism of Hopf algebras,
where $\dot\chi^I(\nu)=\chi^{I}(\nu)/\chi^{I}(\nu)(\bm 0)$, $\comp(I)$ is the composition of $n$ corresponding to $I$, 
and $L_{\comp(I)}$ is the fundamental quasisymmetric function attached to $\comp(I)$
(Theorem~\ref{thm: categorification of QSym}).

The second objective of the present paper is to study  
the basis $\{ \mathcal{B}(q,t)_{\alpha} \mid \alpha \in \Comp \}$ for the Hopf algebra $\NSym_{\mathbb{C}(q,t)}$ of noncommutative symmetric function over $\mathbb{C}(q,t)$ (=the rational function field in commuting variables $q$ and $t$), 
where 
\[
\mathcal{B}(q,t)_{\comp(I)} := \sum_{J: I \cup J = [n-1]} q^{|I \setminus J|} \  t^{|I \cap J|} \  H_{\comp(J)} \in \NSym_{\mathbb{C}(q,t)}
\]
for $I \subseteq [n-1]$ and $H_{\comp(J)}$ is the complete homogeneous noncommutative symmetric function attached to $\comp(J)$.
This basis has some noteworthy properties. 
For example, if $q$ and $t$ are specialized suitably, it interpolates some  well known bases for $\NSym$ as follows: 
\[
\mathcal{B}(1,0)_{\alpha} = H_{\alpha^c},  \qquad \mathcal{B}(-1,1)_{\alpha} = \Lambda_{\alpha^c}, \qquad \mathcal{B}(1,-1)_{\alpha} = E^*_{\alpha^c},
\]
where $H_{\alpha}$, $\Lambda_{\alpha}$, and $E_{\alpha}^*$ are 
the complete homogeneous, the elementary, the dual essential noncommutative symmetric function in $\NSym$, respectively, 
and $\alpha^c$ denotes the complement of $\alpha$
(see~\ref{Subsection: qt basis B for NSym}).
The main result here is the calculation of the structure constants of $\NSym_{\mathbb{C}(q,t)}$ for this basis. 
In other words, we provide explicit expansions of   $\mathcal{B}(q,t)_{\alpha}\,\mathcal{B}(q,t)_{\beta}$ and $\triangle\mathcal{B}(q,t)_{\alpha}$ in this basis.
Let us briefly sketch the calculation process.
We first derive formulas on the structure constants of $(\bigoplus_{n \ge 0} \scf(\mathcal{S}(\mathcal{N}_n(\nu))),\bm{m},\blacktriangle)$ 
for the basis consisting of superclass identifier functions $\kappa_I(\nu)$.
Second, for each $I \subseteq [n-1]$, we let
\[
\Pi(\nu)_{\comp(I)} := \ch_\nu\left(\frac{\kappa_I(\nu)}{(\nu-1)^{|I|}}\right)
\]
and consider the basis 
$\{\Pi(\nu)_{\alpha} \mid \alpha \in \Comp \}$
for $\QSym$.  
Third, via the duality of $\QSym$ and $\NSym$, 
we obtain formulas on the structure constants of $\NSym$ for the dual basis $\{\Pi(\nu)^*_{\alpha} \mid \alpha \in \Comp \}$ of $\{\Pi(\nu)_{\alpha} \mid \alpha \in \Comp \}$ (Lemma~\ref{lem: str consts for Pi}).
And we show that 
\[
\Pi(\nu)^*_{\alpha} = \mathcal{B}(-\nu,\nu-1)_{\alpha}
\]
for every $\alpha \in \Comp$ and $\nu >1$.
Finally, putting these together, we derive formulas on the structure constants of  $\NSym_{\mathbb{C}(q,t)}$ for the basis $\{\mathcal{B}(-q,q-1)_{\alpha} \mid \alpha \in \Comp \}$, which immediately leads us to the desired result
(Theorem~\ref{thm: str consts for B}).
Contrary to the product, the structure constants associated with the coproduct appear as rather complicated  polynomials in $q$ and $t$ with integer coefficients.
However, they can be written in a much simpler form if they are restricted to $\{\mathcal{B}(q,t)_{1^n} \mid n \ge 0 \}$, a generating set of $\NSym_{\mathbb{C}(q,t)}$ (Theorem~\ref{thm: str consts for hat B}).

As applications, we obtain generating sets for $\NSym$ and $\Sym$ from the description of the structure constants associated with the product.
For instance,
\[
\left\{ \sum_{\lambda :  \text{ partitions of $n$ }} a^{n-\ell(\lambda)}  b^{\ell(\lambda)-1}  C_{\lambda}  \, h_{\lambda} \mid n \ge 0 \right\}
\]
is a generating set of $\Sym$ for each $a \in \mathbb{C} \setminus \{0\}$ and $b \in \mathbb{C}$, where $C_{\lambda} = {\ell(\lambda)!}/{\prod_{i} m_i(\lambda)!}$ for $\lambda = 1^{m_1(\lambda)} 2^{m_2(\lambda)} \cdots $ 
and $h_{\lambda}$ is the complete homogeneous symmetric function attached to $\lambda$ (Corollary~\ref{minimal geberating sets for nsym and sym}).
Also, using the description of the structure constants associated with the coproduct, we redescribe the number of overlapping shuffles of two compositions $\alpha$ and $\beta$ with weight $\gamma$ (Corollary~\ref{cor: overlapping shuffle}).

As seen above, our categorifications
provide many interesting applications related to quasisymmetric functions and noncommutative symmetric functions.
We expect our approach to be useful in studying the structure constants of various combinatorial Hopf algebras.

The paper is organized as follows.
In Section~\ref{Section: Preliminaries}, we briefly review normal supercharacter theories and the Hopf algebras $\QSym$, $\NSym$, and $\FQSym$.
In Section~\ref{Section: categorification of QSym}, for each integer $\nu>1$, we consider the normal supercharacter theories 
$\mathcal{N}_n(\nu)$ for $n\ge 0$ and then endow the $\mathbb C$-space 
$\bigoplus_{n \ge 0} \scf(\mathcal{S}(\mathcal{N}_n(\nu)))$ with
the Hopf algebra structure that is isomorphic to $\QSym$.
Furthermore, we compute the structure constants of the Hopf algebra thus obtained 
for the basis consisting of superclass identifier functions.
Section~\ref{Section: new bases for NSym} is devoted to the study of a 
new basis $\{ \mathcal{B}(q,t)_{\alpha} \}$ for $\NSym_{\mathbb{C}(q,t)}$,
with an emphasis on the structure constants of $\NSym_{\mathbb{C}(q,t)}$ for this basis. Some interesting applications are also obtained via the specializations of $q$ and $t$.

\section{Preliminaries}\label{Section: Preliminaries}
Given any integers $m$ and $n$, define $[m,n]$ to be the interval $\{t\in \mathbb Z \mid m\le t \le n\}$ whenever $m \le n$ and the empty set $\emptyset$ else.
For simplicity, we set $[n]:=[1,n]$ and therefore
$[n]=\emptyset$ if $n<1$.
Unless otherwise stated, $n$ will denote a nonnegative integer throughout this paper. 

\subsection{Supercharacter theories of a finite group}\label{subsection: Supercharacter theories of a finite group}
In this subsection, $G$ denotes a finite group.
A supercharacter theory of $G$ is a variation of the character theory of $G$. 
Given a set partition $\mathcal{K}$ of $G$, let $f(G;\mathcal{K})$ 
be the $\mathbb C$-vector space consisting of functions which are constant on the blocks in $\mathcal{K}$, that is, 
\[
f(G;\mathcal{K}):= \{\psi \mid G \rightarrow \mathbb{C} \mid \psi(g) = \psi(k) \text{ if $g$ and $h$ are in the same block in $\mathcal K$}\}.
\]
And, we let $\Irr(G)$ be the set of irreducible characters of $G$.

\begin{definition}{\rm (\cite{DI08})}\label{def: supercharacter theory}
A \emph{supercharacter theory} $\mathcal{S}$ of $G$ is a pair $(\cl(\mathcal{S}) , \ch(\mathcal{S}))$, where $\cl(\mathcal{S})$ is a set partition of $G$ and $\ch(\mathcal{S})$ 
is a set partition of $\Irr(G)$, such that
\begin{enumerate}
    \item[{\bf C1.}] $\{e\} \in \cl(\mathcal{S}) $,
    \item[{\bf C2.}] $|\cl(\mathcal{S}) | = |\ch(\mathcal{S})|,$
    \item[{\bf C3.}] $\text{For each block $X$ in $\ch(\mathcal{S})$}$,
    \[
    \chi^X:=\sum_{\psi \in X} \psi(e) \psi \in f(G; \cl(\mathcal{S})),
    \]
    where $e$ denotes the identity of $G$.
\end{enumerate}
\end{definition}

Each block in $\cl(\mathcal{S})$ is called a \emph{superclass} of $\cl(\mathcal{S})$ and, for each block $X$ of $\ch(\mathcal{S})$,
$\chi^X$ is called a \emph{supercharacter} of $\mathcal{S}$.
And, $f(G;\cl(\mathcal{S}))$ is called the \emph{supercharacter function space} of $\mathcal{S}$, denoted by $\scf(\mathcal{S})$.

Perhaps the most familiar examples of supercharacter theories are 
\begin{align*}
 & (\{\text{conjugacy classes of $G$}\}, \{ \{\psi \} \mid \psi \in \Irr(G) \}) \text{ and }\\
 &(\{ \{e\}, G \setminus \{ e \} \}, \{ \{\mathbbm{1} \}, \Irr(G) \setminus \{\mathbbm{1} \}  \}),
\end{align*}
where $\mathbbm{1}$ is the trivial character of $G$.

For each $I \in \cl(\mathcal{S})$, consider the function $\kappa_I \in f(G;\cl(\mathcal{S}))$
defined by 
\[
\kappa_I(g) =
\begin{cases}
      1 &\text{if } g \in I, \\
      0 &\text{otherwise,}
\end{cases}
\]
which is called the \emph{superclass identifier function} attached to $I$.
Combining the orthogonality of irreducible characters with the conditions {\bf C2} and {\bf C3}, 
one can easily see that 
$\{\kappa_I \mid I \in \cl(\mathcal{S}) \}$ and $\{\chi^X \mid X \in \ch(\mathcal{S}) \}$
are $\mathbb C$-bases of $\scf(\mathcal{S})$.

Supercharacter theories can be generated in many ways.
Here we introduce the normal supercharacter theory introduced by Aliniaeifard~\cite{A17}.
Let $\ker(G):=\{N \trianglelefteq G \}$ be the lattice of normal subgroups of $G$ ordered by inclusion. 
For $M,N\in \ker(G)$, the meet and join of $M$ and $N$ are given by
\[
M \vee N = MN, \quad M \wedge N = M \cap N.
\]
Define a \emph{sublattice} $\mathcal{N}$ of $\ker(G)$ by a subset of $\ker(G)$ such that
\begin{enumerate}
    \item $\{e\}, G \in \mathcal{N}$ and 
    \item $\mathcal{N}$ is closed under meet and join operations.
\end{enumerate}
Obviously every sublattice also forms a lattice under $\vee, \wedge$.
Given $N \in \mathcal{N}$, let
\[
\mathrm{C}(L) := \{ O \in \mathcal{N} \mid O \text{ covers } L\}
\]
and 
\begin{align*}
N_{\circ} &:= \{g \in N \mid g \notin M \text{ for all } M\in \mathcal{N} \text{ with } N \in {\rm C}(M)\} \text{ and }\\
X^{N^{\bullet}} &:= \{\psi \in \Irr(G) \mid N \subseteq \ker(\psi) , \text{ but } O \nsubseteq \ker(\psi) \text{ for all }O\in C(N)\}.
\end{align*}
With this notation, the following theorem is proved in \cite{A17}. 
\begin{theorem}{\rm (\cite[Theorem 3.4]{A17})}\label{def of normal character theory} 
Given a sublattice $\mathcal{N}$ of $\ker(G)$, 
let
\begin{align*}
  &\cl(\mathcal{S} (\mathcal{N})) := \{N_{\circ} \mid N \in \mathcal{N} \text{ and } N_{\circ} \neq \emptyset \}\\
  &\ch(\mathcal{S} (\mathcal{N})) := \{ X^{N^{\bullet}} \mid N \in \mathcal{N} \text{ and } X^{N^{\bullet}} \neq \emptyset\}.
    \end{align*}
Then $(\cl(\mathcal{S} (\mathcal{N})),\ch(\mathcal{S}(\mathcal{N})))$ 
defines a supercharacter theory $\mathcal{S}(\mathcal{N})$ of $G$.
\end{theorem}

For each sublattice  $\mathcal{N} \subseteq \ker(G)$, the supercharacter theory $\mathcal{S}(\mathcal{N})$
is called a \emph{normal supercharacter theory} of $G$.
If there is no danger of confusion, we simply write $\mathcal{N}$ for $\mathcal{S}(\mathcal{N})$.
Using normal supercharacter theories, Aliniaeifard and Thiem~\cite{AT21, AT21-Nsym} successfully categorify the Hopf algebra $\FQSym$ and the Hopf algebra $\NSym$.
In Section~\ref{Section: categorification of QSym}, we present a categorification of the Hopf algebra $\QSym$ using normal supercharacter theories of certain finite abelian groups.

\subsection{The Hopf algebras in our consideration}
\label{Hopf algebras in consideration}
We start by recalling the definitions and properties of Hopf algebras.
All of these are borrowed unchanged from \cite{GR20}.

A Hopf algebra is a bialgebra $H$ over $\mathbb{C}$ together with a $\mathbb{C}$-linear map $S: H \rightarrow H$, called the antipode, which satisfy certain compatible relations.
If $A = \bigoplus_{n \ge 0} A_n$ is a $\mathbb Z_{\ge 0}$-graded $\mathbb{C}$-algebra with $A_0 \cong \mathbb{C}$, then $A$ is said to be \emph{connected}. 
It is well known that every connected graded bialgebra $H$ has a unique antipode $S$ endowing it with 
a Hopf structure (for instance, see~\cite[Proposition 1.4.14]{GR20}).

In the present paper, we deal with Hopf algebras $\QSym$, $\NSym$, and $\FQSym$,
which were firstly introduced in~\cite{gessel84},~\cite{Gelfand95}, and~\cite{MR95}, respectively.
Since all of them are connected $\mathbb Z_{\ge 0}$-graded bialgebras, 
we usually do not mention their antipodes unless otherwise stated. 
Before introducing these algebras, let us collect the necessary basic definitions and notation.

A composition is a finite tuple $\alpha = (\alpha_1, \alpha_2,\ldots,\alpha_l)$ of positive integers. Its length is defined to be $l$ and denoted by $\ell(\alpha)$ and its size is defined to be $\alpha_1 + \alpha_2 +\cdots + \alpha_l$ and denoted by $|\alpha|$.
Denote the set of composition of $n$ by $\Comp_n$.
A partition $\lambda = (\lambda_1, \lambda_2, \ldots, \lambda_l)$ is a composition with additional condition that $\lambda_1 \ge \lambda_2 \ge \cdots \ge \lambda_l$. Denote the set of partition of $n$ by $\Par_n$.
A composition $\alpha$ of $n$ determines a partition by rearranging in order of $|\alpha_i|$. Denote it by $\lambda(\alpha)$.

For each positive integer $n$, 
there is a 1-1 correspondence between $\Comp_n$ and subsets of $[n-1]$ given by 
\begin{align*}
     \alpha=(\alpha_1,\alpha_2,\ldots,\alpha_l) &\mapsto  {\rm set}(\alpha):=\{ \alpha_1, \alpha_1 + \alpha_2, \ldots, \alpha_1+\alpha_2+\ldots+\alpha_{l-1} \}\\
     S=\{s_1<s_2< \cdots< s_i\} &\mapsto {\rm comp}(S):= (s_1,s_2-s_1, \ldots,s_i-s_{i-1}, n-s_i). 
\end{align*}

\subsubsection{The Hopf algebra of quasisymmetric functions}
Let $x=(x_1, x_2, \ldots )$ be the infinite totally ordered set of commuting variables, and let $\mathbb{C}[[x_1,x_2,\ldots]]$ be the algebra of formal power series of bounded degree. 
For each composition 
$\alpha=(\alpha_1,\alpha_2,\ldots,\alpha_l)$, define the {\it monomial quasisymmetric function} to be
$$
M_{\alpha}= \sum_{j_1<j_2<\cdots<j_l} x_{j_1}^{\alpha_1}x_{j_2}^{\alpha_2} \cdots x_{j_{l}}^{\alpha_l}.
$$
The algebra $\QSym$ of \emph{quasisymmetric functions} over $\mathbb{C}$ is defined by 
$$
\QSym := \bigoplus_{n \ge 0} \QSym_n \subseteq \mathbb{C}[[x_1,x_2,\ldots]],
$$
where $\QSym_n := {\rm span}_{\mathbb{C}} \{ M_{\alpha} \mid \alpha \in \Comp_n \}$.

For $\alpha, \beta \in \Comp_n$, we say that $\alpha$ refines $\beta$ if one can obtain $\beta$ from $\alpha$ by combining some of its adjacent parts. Alternatively, this means that $\set(\beta) \subseteq \set(\alpha)$. Denote this by $\alpha \preceq \beta$.
With this ordering, we have two other bases $\{F_{\alpha}\}$ and $\{E_{\alpha}\}$, where 
\begin{align*} 
&F_{\alpha}=\sum_{\beta \succeq \alpha} M_{\beta} \quad (\text{the \emph {fundamental quasisymmetric function}})\\
&E_{\alpha}=\sum_{\alpha \succeq \beta} M_{\beta} \quad (\text{the \emph {essential quasisymmetric function}}).
\end{align*}
The latter basis was investigated intensively by Hoffman in~\cite{H15}.

The algebra $\QSym$ has a natural coproduct structure. 
In particular, in the basis of monomial quasisymmetric functions, the coproduct formula can be expressed 
in the following form:
$$
\triangle M_{\alpha}= \sum_{\beta \cdot \gamma = \alpha} M_{\beta} \otimes M_{\alpha},
$$
where $\beta \cdot \gamma$ is the concatenation of compositions $\beta$ and $\gamma$.
For reference, in this case, the antipode is given as follows:
\[
S(M_{\alpha}) = (-1)^{\ell(\alpha)} \sum_{\gamma \succeq \alpha^r} M_{\gamma},
\]
where $\alpha^r= (\alpha_l,\alpha_{l-1},\ldots,\alpha_1)$ is the reverse composition of $\alpha = (\alpha_1,\alpha_2,\ldots,\alpha_l)$.

\subsubsection{The Hopf algebra of noncommutative symmetric functions}
The Hopf algebra $ \NSym $ of \emph{noncommutative symmetric functions} over $\mathbb{C}$ is defined by the graded dual of $\QSym$, i.e.,
$\NSym = \QSym^\circ. $
Let 
\[(\cdot,\cdot) : \NSym \otimes \QSym \rightarrow \mathbb{C}\]
be the dual pairing and $H_{\alpha}$ be the dual basis of $M_{\alpha}$ so that
$(H_{\alpha}, M_{\beta}) = \delta_{\alpha, \beta}$.
Letting $H_n := H_{(n)}$ for $n=1,2,\ldots$, one can see that $\NSym \simeq \mathbb{C} \langle H_1,H_2,\ldots \rangle$ is a free associative algebra on generators $\{H_1, H_2, \ldots \}$ with the coproduct determined by
$$
\triangle H_n= \sum_{i+j=n} H_i \otimes H_j
$$
(see~\cite[Thm5.4.2]{GR20}).

The noncommutative symmetric function $H_{\alpha}$ $(\alpha \in \Comp)$ are called the \emph{noncommutative complete homogeneous symmetric functions}.
There are many well known bases for $\NSym$ other than $\{H_{\alpha}\}$. 
We are particularly interested in the bases  $\{\Lambda_{\alpha}\}$, $\{R_{\alpha}\}$, $\{E^*_{\alpha}\}$, where 
\begin{align*}
 \Lambda_{\alpha} &:= \sum_{\beta \preceq \alpha} (-1)^{n-\ell(\beta)} H_{\beta} 
 \quad (\text{the \emph{noncommutative elementary symmetric function}})\\
  R_{\alpha} &:= \sum_{\alpha \preceq \beta} (-1)^{\ell(\beta)-\ell(\alpha)} H_{\beta} 
  \quad (\text{the \emph{noncommutative ribbon Schur function}})\\
  E^*_{\alpha} &:= \sum_{\beta \preceq \alpha} (-1)^{\ell(\beta)-\ell(\alpha)} H_{\beta}
   \quad (\text{the \emph{dual essential noncommutative symmetric function}})
\end{align*}
In fact, $\{R_{\alpha}\}$ is the dual basis of $\{F_{\alpha}\}$ and $\{E^*_{\alpha}\}$ is the dual basis of $\{E_{\alpha}\}$ with respect to the dual pairing $(\cdot , \cdot )$.

\begin{theorem} {\rm (\cite[Corollary 5.4.3]{GR20})} \label{thm: comm}
Let $\Sym$ be the Hopf algebra of symmetric functions over $\mathbb{C}$ and $h_n$ be the $n$th complete homogeneous symmetric function.
The algebra homomorphism
\begin{equation}\label{def of commutative image}
   {\rm comm}: \NSym \to \Sym , \quad H_n \mapsto h_n 
\end{equation}
is a surjective Hopf algebra homomorphism. And, ${\rm comm}(\Lambda_n) = e_n$ where $e_n$ is the $n$th elementary symmetric function.
\end{theorem}

\subsubsection{The Hopf algebra of free quasisymmetric functions}

we begin by introducing the necessary notation.
For positive integers $m$ and $n$, let 
\begin{equation}\label{notation for the set of shuffles}
\textsf{Sh}_{m,n}:=\left\{\sigma \in \SG_{m+n} \mid \substack{\sigma^{-1}(1) < \sigma^{-1}(2) < \cdots < \sigma^{-1}(m)\quad \text{ and }\\ \sigma^{-1}(m+1) < \sigma^{-1}(m+2) < \cdots < \sigma^{-1}(m+n)} \right\}.
\end{equation}
For words $u = u_1 u_2  \ldots u_m$, $v = v_1 v_2  \ldots v_n$, and a permutation $\sigma \in \textsf{Sh}_{m,n}$, 
set
\begin{align*}
u \shuffle_{\sigma} v := w_1 w_2 \ldots w_{m+n},
\end{align*}
where $w_{\sigma^{-1}(i)}=u_i$ for $1\le i \le m$ and $w_{\sigma^{-1}(m+j)}= v_{j}$ 
for $1\le j \le n$.
Then we define  
\[u \shuffle v := \{ u \shuffle_{\sigma} v \mid \sigma \in \textsf{Sh}_{m,n} \} \quad (\text{as a multiset}).\]
Using one line notation, let us identify each permutation $\pi \in \SG_n$ with the word $\pi_1 \pi_2 \ldots \pi_n$, where $\pi_i = \pi(i)$ for $i \in [n]$.
The \emph{$m$-shifted permutation} of $\pi$ is defined by the permutation $\pi[m] \in \SG_{[m+1, m+n]}$ corresponding to the word $(\pi_1+m) (\pi_2+m) \ldots (\pi_n+m)$.
For instance, $21[2]=43$ and therefore    
\[
12\,\shuffle \,{21[2]}= \{12 {43}, 1{4}2{3}, 1{4}{3}2, {4}12{3},{4}1 { 3}2, {43}{12}\}.
\]
\begin{definition}{\rm (\cite[Theorem 3.3]{MR95})}
The Hopf algebra $\FQSym$ of free quasisymmetric functions is defined to be the graded Hopf algebra over $\mathbb{C}$
$$
\FQSym := \bigoplus_{n \ge 0} \FQSym_n,
$$
where $\FQSym_n := {\rm span}_{\mathbb{C}} \{ F_w \mid w\in \SG_n\}$, 
with the following product and coproduct:
\begin{itemize}
\item
For each $u \in \SG_m$ and $v \in \SG_n$, the product is defined by
$$
F_u F_v := \sum_{w \in u \shuffle v[m]} F_w,
$$
\item
and, for each $w \in \SG_n$, the coproduct is defined by
$$
\triangle F_w := \sum_{k=0}^{n} F_{\std(w_1w_2 \ldots w_k)} \otimes F_{\std(w_{k+1}w_{k+2}\ldots w_n)},
$$
where $\std(w)$ stands for the standardization of $w$. 

\end{itemize}
\end{definition}

\begin{example}
 $\triangle(F_{132})= 1\otimes F_{132} + F_{1} \otimes F_{21} + F_{12} \otimes F_{1} + F_{132} \otimes 1$.
\end{example}

\begin{theorem} {\rm (\cite[Cororally 8.1.14]{GR20})}\label{thm: FQSym to QSym}
The $\mathbb{C}$-linear map
\begin{align*}
  \pi: \FQSym \to \QSym, \quad F_w \mapsto L_{\comp(\Des(w))} \quad (w \in \SG_n, \, n\ge 0)
\end{align*}
is a surjective Hopf algebra homomorphism, where $\Des(w)$ is the descent set of $w$, i.e.,
$$
\Des(w) = \{ i \mid w(i) > w(i+1)\} \subseteq [n-1].
$$
\end{theorem}

The diagram below shows the relationship between the Hope algebras discussed so far.
The morphisms $\pi^*$ and ${\rm comm}^*$ represent the dual maps of $\pi$ and ${\rm comm}$, respectively, 
and the unlabelled arrow in the middle represents the duality of the Hopf algebras.

\[
\begin{tikzcd}[row sep=scriptsize, column sep=scriptsize]
& \FQSym \arrow[dl, twoheadrightarrow, "\pi"] &  \\
\QSym \arrow[rr, leftrightarrow] & & \NSym  \arrow[ul, hook', "\pi^*"] \arrow[dl, twoheadrightarrow, "{\rm comm}"] \\
& \Sym  \arrow[ul, hook', "{\rm comm}^*"] & & \\
\end{tikzcd}
\]

\section{Categorifications of $\QSym$ using supercharacter theories}
\label{Section: categorification of QSym}

We assume that $\nu$ is a positive integer $>1$, which will be fixed throughout this section.

\subsection{Normal supercharacter theories of $\bigoplus C_\nu$}
\label{subsection: Normal supercharacter theories of finite abelian groups}
In~\cite{AT21}, Aliniaeifard and Thiem categorified $\FQSym$ via towers of groups and their supercharacter theories, and our approach here is basically based on this paper.

For a cyclic group $C_\nu$ of order $\nu$ and a finite set $S$, let 
$$
Q_S(\nu) := \bigoplus_{s \in S} C_{\nu,s}, 
$$
where $C_{\nu,s} = C_\nu$ for all $s \in S$. 
For clarity, we use $0$ to represent the additive identity of $C_\nu$ and $\bm 0$ to represent the identity of $Q_S(\nu)$.
For each subset $I$ of $S$, we identify $Q_I(\nu)$ with the subgroup of $Q_S(\nu)$ whose 
$j$th component is $C_\nu$ if $j\in I$ and $\{0\}$ else.
Let
\[\mathcal{N}_S(\nu):= \{ Q_I(\nu) \mid I \subseteq S \}.\]
Under this identification, it can be easily seen that  
$\mathcal{N}_S(\nu)$ is a sublattice of $Q_S(\nu)$ and thus 
gives rise to a normal supercharacter theory $\mathcal{S}(\mathcal{N}_S(\nu))$ of $Q_S(\nu)$ 
(see  Theorem~\ref{def of normal character theory}).
For simplicity, we write 
$Q_n(\nu)$ and $\mathcal{N}_n(\nu)$ for $Q_{[n-1]}(\nu)$ and $\mathcal{N}_{[n-1]}(\nu)$, respectively.

The sublattice $\mathcal{N}_S(\nu)$ form a distributive lattice which implies that $N_{\circ} \neq \emptyset$ and $X^{N^{\bullet}} \neq \emptyset$ by~\cite[Corollary 3.11]{AT20}. Therefore, all blocks of $\cl(\mathcal{S}(\mathcal{N}_n(\nu)))$
and $\ch(\mathcal{S}(\mathcal{N}_n(\nu)))$ are parametrized by subsets of $[n-1]$. 
This implies that the superclass function space,  denoted by $\scf(\mathcal{S}(\mathcal{N}_n(\nu)))$, is of dimension $|\{ I \subseteq [n-1]  \}| = 2^{n-1}$, so it has the same dimension as $\QSym_n$. 

For each $I \subseteq S$, set
\begin{equation}\label{def of two notable bases}
\begin{aligned}
 &\cl_I(\nu):= {Q_I(\nu)}_{\circ}    \quad     \,\text{ and } \quad  \ch_I(\nu):= X^{{Q_I(\nu)}^{\bullet}}\\
&\kappa_I(\nu):=\kappa_{\cl_I(\nu)} \qquad \text{ and } \quad  \chi^I(\nu):=\chi^{\ch_I(\nu)}. 
\end{aligned}
\end{equation}

\begin{remark}\label{ambiguity of notation}
It should be noted that the notation in~\eqref{def of two notable bases} depends $S$ as well as $I$.
If necessary, we will clarify $S$ as in Definition~\eqref{def: product of scf}.
\end{remark}

From the definition of $Q_I(\nu)_{\circ}$ it follows that 
\[
\cl_I(\nu) = \{(g_i)_{i \in S} \in Q_n(\nu) \mid g_i \text{ is nonzero if $i \in I$ and zero else}\}.
\] 
In the following, we provide a formula for $\chi^I(\nu)$.
Let $\mathbbm{1}$ be the trivial character of $C_\nu$ and $\mathbbm{reg}$ be the character of the regular representation of $C_\nu$.

\begin{proposition}\label{thm: supercharacter calculation}
Let $I \subseteq S$.
For $\bm{g} = (g_i)_{i \in S} \in Q_S(\nu)$, 
we have 
$$
\chi^I(\nu) (\bm{g}) = \prod_{i \in I} \mathbbm{1}(g_i) \prod_{j \in I^c} (\mathbbm{reg} - \mathbbm{1}) (g_j).
$$
\end{proposition}

\begin{proof}
By~\cite[Corollary 3.4, 3.5]{AT20}, for $\bm{g} \in \cl_J(\nu)$, we have
\begin{align*}
    \chi^I(\nu) (\bm{g}) &= 
\begin{cases}
    \dfrac{\nu^{n-1}}{\nu^{|I|}} \left(\dfrac{\nu-1}{\nu}\right)^{|\text{covers of } I|}  \left(\dfrac{1}{1-\nu}\right)^{|J \setminus I|}  & \text{ if } I <_{c} J,\\
    0              & \text{otherwise}
\end{cases}\\
    &=
\begin{cases}
    (\nu-1)^{|I^c|} \
    \left(\dfrac{-1}{\nu-1}\right)^{|J \setminus I|} & \text{ if } I <_{c} J,\\
    0              & \text{ otherwise},
\end{cases}
\end{align*}
where $<_c$ represents the covering relation of the poset $\{I: I \subseteq S \}$ ordered by inclusion.
Now, the assertion follows from the fact that $\mathbbm{1}(g) = 1$ for all $g \in G$ and
\[
(\mathbbm{reg} - \mathbbm{1}) (g) =
\begin{cases}
    \nu-1 & \text{ if } g = 0, \\
    -1  & \text{ otherwise.}
\end{cases}
\]
\end{proof}

It is convenient to express $\chi^I(\nu)$ in the form of coordinates as in~\cite{AT21}.
For a finite group $G$, let $\cf(G)$ by denotes the $\mathbb C$-vector space of the {\it class functions} on $G$.
Consider the natural isomorphism of vector spaces
\begin{equation*}
    \slb \,\, \cdot \,\, \srb: \bigotimes_{s \in S} \cf(C_{\nu,s})\to \cf(Q_S(\nu))
\end{equation*}
defined by $\lb \bigotimes_{i \in S} \phi_i \rb (\bm{g}) = \prod_{i\in S}\phi_i(g_i)$ for $\phi_i \in \cf(C_{\nu,i})$ and $\bm{g} = (g_i)_{i \in S} \in Q_S(\nu)$.

\smallskip
{\bf Convention.} 
When $ S=\{s_1<s_2< \cdots< s_i\}$, we 
write $\bigotimes_{i \in S} \phi_i$ and $\lb \bigotimes_{i \in S} \phi_i \rb$ as $(\phi_{s_1},\phi_{s_2},\ldots,\phi_{s_i})$
and $\lb \phi_{s_1},\phi_{s_2},\ldots,\phi_{s_i} \rb$, respectively.
\smallskip

For each $i\in S$, let   
\begin{align*}
    {\chi^I(\nu)}_{i} := 
\begin{cases}
    \mathbbm{1}  &   \text{if } i \in I,\\
    \mathbbm{reg-1}  & \text{if } i \in S\setminus I.
\end{cases}
\end{align*}
It follows from Proposition~\ref{thm: supercharacter calculation} that \[\chi^I(\nu)=\lb(\chi^I(\nu)_{i})_{i\in S} \rb.\]
For example, if $I = \{ 1,6 \} \subseteq S=\{1,2,4,6,7\}$, then
\[\chi^I(\nu) = \lb \mathbbm{1}, \mathbbm{reg-1}, \underline{\hspace{0.5cm}},\mathbbm{reg-1}, \underline{\hspace{0.5cm}}, \mathbbm{1}, \mathbbm{reg-1} \rb.\]

\subsection{A Hopf algebra structure of $\bigoplus_{n \ge 0} \scf(\mathcal{N}_n(\nu))$}
Hereafter we will simply write  $\scf(\mathcal{N}_n(\nu))$ for $\scf(\mathcal S(\mathcal{N}_n(\nu)))$, the supercharacter function space of $\mathcal S(\mathcal{N}_n(\nu))$. 

Given a subset $S=\{s_1<s_2< \cdots < s_t\}$ of $[n-1]$ with $t:=|S|$,  
consider the {\emph {standardization}}, i.e., the group isomorphism
\begin{align*}
    \iota_S: Q_S(\nu) \to Q_{t+1}(\nu), \quad 
    (g_{s})_{s\in S} \mapsto (g_{s_i})_{1\le i \le t}.
\end{align*}
This induces a $\mathbb C$-linear isomorphism 
\begin{align*}
    {{\iota}_S}^\ast  : \cf(Q_{t+1}(\nu)) \to \cf(Q_S(\nu)), \quad 
    \phi \mapsto \phi \circ  {\rm stan}.
\end{align*}
For simplicity, we set $\iota_S:= {{\rm stan}_S}^\ast$.

Given a subset $S=\{s_1<s_2< \cdots < s_t\}$ of $[n-1]$ with $t:=|S|$,  
consider the group isomorphism 
\begin{align*}
    \iota_S: Q_S(\nu) \to Q_{t+1}(\nu), \quad 
    (g_{s})_{s\in S} \mapsto (g_{s_i})_{1\le i \le t},
\end{align*}
which is obtained by the standization of indices. 
This induces a $\mathbb C$-linear isomorphism 
\begin{align*}
    {{\iota}_S}^\ast  : \cf(Q_{t+1}(\nu)) \to \cf(Q_S(\nu)), \quad 
    \phi \mapsto \phi \circ {\iota}_S.
\end{align*}
For $I\subseteq [t]$, let $S_I := \{ s_i \mid i \in I \}$.
Since ${{\iota}_S}^\ast(\chi^I(\nu)) = \chi^{S_I}(\nu)$,
this isomorphism restricts to the following isomorphism of the supercharacter function spaces:
\begin{align*}
    {{\iota}_S}^\ast : \scf(\mathcal{N}_{t+1}(\nu)) \to \scf(\mathcal{N}_S(\nu)), \quad 
    \chi^I(\nu) \mapsto \chi^{S_I}(\nu).
\end{align*}
Let us identify $Q_S(\nu) \times Q_{S^c}(\nu)$ with $Q_{[n-1]}(\nu)$ in the natural way.
For $\phi \in \cf(Q_S(\nu))$ and $\psi \in \cf(Q_{S^c}(\nu))$, let $\phi \otimes_S \psi$ be the class function of $Q_{[n-1]}(\nu)$ defined by
\begin{align*}
    (\phi \otimes_S \psi) (a,b) = \phi(a) \psi(b),
\end{align*}
where $(a,b) \in Q_S(\nu) \times Q_{S^c}(\nu) = Q_{[n-1]}(\nu)$. 
It is easy to see that if $\phi \in \scf(\mathcal{N}_S(\nu))$ and $\psi \in \scf(\mathcal{N}_{S^c}(\nu))$ then $(\phi \otimes_S \psi) \in \scf(\mathcal{N}_n(\nu))$.
In case where $S = [n-2]$, or equivalently $S^c = \{n-1 \}$, 
we simply write $\otimes_1$ for $\otimes_S$.

Let $A$ be a subset of $[n]$. 
A subset of $A$ is said to be \emph{connected $($in $A)$} if it consists of consecutive integers or a single element,
and \emph{maximally connected $($in $A)$} if it is maximal among connected subsets
with respect to inclusion order.
Let 
\[\textsf{conn}(A):=\{\text{maximally connected subsets of $A$}\}.\] 
Label the subsets in $\textsf{conn}(A)$ as $A_1, A_2, A_3, \ldots$, where  
$\min A_i < \min A_j$ if $i<j$.
Let 
\begin{align*}
& c_1(A) : = \{ \max B \mid B \in \textsf{conn}(A)  \} \setminus \{ n \},\\ 
& c_2(A) := \{ \max B  \mid B \in \textsf{conn}(A^c) \} \setminus \{ n \},\\
& c(A) := c_1(A) \cup c_2(A).
\end{align*}
\begin{example}
Let $n=9$ and $A = \{1,2,5,7,8,9\} \subseteq [9]$. 
Then $A_1 = \{1,2\}$, $A_2 =\{5\}$, $A_3 =\{7,8,9\}$ are all maximally connected subsets of $A$.
Moreover, $c_1(A) = \{2, 5 \}$, $c_2(A) = \{4, 6 \}$, and thus $c(A) = \{2,4,5,6 \}$.
\end{example}

In what follows, we will define a product and a coproduct on $\bigoplus_{n \ge 0} \scf(\mathcal{N}_n(\nu))$. 
To do this, we need the following notations.
\begin{itemize}
\item 
Let $G$ be a group and $H$ a subgroup of $G$. 
For $\phi \in \cf(G)$, denote by $\phi\downarrow^{G}_{H}$
the restriction of $\phi$ from $G$ to $H$.

\item
For positive integers $a$ and $b$, 
denote by $\binom{[a]}{b}$ the collection of subsets of $[a]$ with size $b$.

\item
For $I \subseteq S$, set 
$\dot\chi^I(\nu):=\dfrac{\chi^{I}(\nu)}{\chi^{I}(\nu)(\bm 0)}$.

\item
Let $m,n\ge 0$ and $A \in \binom{[m+n]}{n}$.
For $\phi \in \cf(Q_m(\nu))$ and $\psi \in \cf(Q_n(\nu))$,
set
\begin{equation}\label{def of s(phi,psi)}
{\bf s}_A(\phi, \psi) := 
\left({{\iota}_{A^c}}^\ast\left(\phi \otimes_1 \dfrac{\mathbbm{reg - 1}}{\nu-1}\right)\right) \otimes_{A^c} \left({{\iota}_{A}}^\ast\left(\psi \otimes_1 \dfrac{\mathbbm{reg - 1}}{\nu-1} \right)\right).
\end{equation}
Here, we are viewing $\otimes_{A^c}$ as a map
\begin{align*}
\otimes_{A^c}  : \cf(Q_{A^c}(\nu)) \times \cf(Q_{[m+n] \setminus A^c}(\nu)) \ra \cf(Q_{m+n+1}(\nu)).
\end{align*}
\end{itemize}

\begin{definition}\label{def: product of scf}
Let $m$ and $n$ be nonnegative integers.
\begin{enumerate}[label = {\rm (\alph*)}]
\item
For $A \in \binom{[m+n]}{n}$, let  
\[{\bf m}_A:\cf(Q_m(\nu)) \times \cf(Q_n(\nu)) \to \cf(Q_{m+n}(\nu))\]
be the $\mathbb C$-bilinear map given by 
\[
{\bf m}_A(\phi, \psi) 
=
\begin{cases}
\phi \psi &\text{ if } m = 0 \text{ or } n =0,\\
\dot\chi^{c_1(A)}(\nu) \otimes_{c(A)} 
\left( {\bf s}_A(\phi, \psi)  \big\downarrow^{Q_{m+n+1}(\nu)}_{Q_{[m+n-1] \setminus c(A)}(\nu)} \right) &\text{ otherwise}
\end{cases}
\]
for $\phi \in \cf(Q_m(\nu))$ and $\psi \in \cf(Q_n(\nu))$.
Here, we are viewing $\dot\chi^{c_1(A)}(\nu)$ as an element of 
$\cf(Q_A(\nu))$ and $\otimes_{c(A)}$ as a map
\begin{align*}
\otimes_{c(A)} : \cf(Q_{c(A)}(\nu)) \times \cf(Q_{[m+n-1] \setminus c(A)}(\nu)) \ra \cf(Q_{m+n}(\nu)). 
\end{align*}

\item
We define 
\[{\bf m}:\cf(Q_m(\nu)) \times \cf(Q_n(\nu)) \to \cf(Q_{m+n}(\nu))\] 
by
$$
{\bf m}:=
\sum_{A \in \binom{[m+n]}{n}} {\bf m}_A.
$$
\end{enumerate}
\end{definition}

For $n \ge 2$ and $\phi \in \cf(Q_n(\nu))$ and $1 \ge k \ge n-1$,
let $(\phi_1^k, \phi_2^k)$ be a pair satisfying that 
\begin{align*}
&\phi_1^k \in \cf(Q_{[1,k-1]}(\nu)),\\
&\phi_2^k \in \cf(Q_{[k+1,n-1]}(\nu)), \text{ and }\\
&\phi \big\downarrow^{Q_n(\nu)}_{Q_{[1,k-1] \sqcup [k+1,n-1] }(\nu)} = \phi_1^k \otimes_{[1,k-1]} \phi_2^k.
\end{align*}

\begin{definition}\label{def: coproduct of scf}
Let $n$ be a nonnegative integer.

\begin{enumerate}[label = {\rm (\alph*)}]

\item
For $0\le k \le n$, let 
\[\blacktriangle_k:\cf(Q_{n}(\nu))\to \cf(Q_k(\nu)) \otimes \cf(Q_{n-k}(\nu))\]
be the $\mathbb C$-linear map given by
\begin{equation}\label{eq: def of kth coproduct}
\blacktriangle_k(\phi)
=\begin{cases}
\mathbbm{1}_0 \otimes \phi & \text{ if } k=0, \\
\phi_1^k \otimes (\iota^\ast_{[k+1,n-1]})^{-1} (\phi_2^k) \quad & \text{ if } 1\le k< n, \\
\phi \otimes \mathbbm{1}_0 & \text{ if } k=n, 
\end{cases}
\end{equation}
where $\phi \in \cf(Q_n(\nu))$.

\item
We define 
\[\blacktriangle:\cf(Q_{n}(\nu))\to \bigoplus_{k\in [n-1]}\cf(Q_k(\nu)) \otimes \cf(Q_{n-k}(\nu))\]
by 
$$
\blacktriangle := 
\sum_{k=0}^n \blacktriangle_k.
$$
\end{enumerate}
\end{definition}

\begin{remark}
It should be remarked that the pair $(\phi_1^k, \phi_2^k)$ is not unique, but the value of $\blacktriangle_k(\phi)$ does not depend on the choices of this pair.
\end{remark}

Extending ${\bf m}$ bi-additively
and $\blacktriangle$ additively, 
we obtain a $\mathbb C$-bilinear map
\begin{align*}
{\bf m}:\bigoplus_{n \ge 0} \cf(Q_n(\nu))\otimes
\bigoplus_{n \ge 0}\cf(Q_n(\nu)) \to \bigoplus_{n \ge 0}\cf(Q_n(\nu))
\end{align*}
and a $\mathbb C$-linear map
\begin{align*}
\blacktriangle:\bigoplus_{n \ge 0}\cf(Q_n(\nu)) \to \bigoplus_{n \ge 0}\cf(Q_n(\nu)) \otimes \bigoplus_{n \ge 0} \cf(Q_n(\nu)).
\end{align*}

In the following, we will show that ${\bf m}$ and $\blacktriangle$ restrict to the supercharacter function spaces.
For this purpose, we first introduce necessary notation.

\begin{definition}\label{def: I preshuffle J}
Let $m$ and $n$ be nonnegative integers
and let $A \in \binom{[m+n]}{n}$, $I \subseteq [m-1]$, and $J \subseteq [n-1]$.   

\begin{enumerate}[label = {\rm (\alph*)}]
\item
Let 
\begin{align*}
    (I \#_A J)' := (A^c)_{I'} \sqcup A_{J'} \quad (\subseteq [m+n]).
\end{align*}
Here $I'=I, J'=J$, but we are viewing them as $I' \subseteq [m]$ and $J' \subseteq [n]$.

\item
The \emph{$A$-preshuffle} $I \#_A J$ of $I$ and $J$ is defined by the subset of $[m+n-1]$ 
with the same elements as $(I \#_A J)'$.

\item 
Let
$$I \shuffle_A J := c_1(A) \sqcup ((I \#_A J) \setminus c(A)).$$
\end{enumerate}
\end{definition}
For $S =\{s_1, s_2, \ldots, s_t\} \subseteq [k+1, n-1]$, let
\[
S-k:=\{s_1 -k, s_2 -k, \ldots, s_t -k \} \quad (\subseteq [n-k+1]).
\]
\begin{lemma}\label{thm: categorification of QSym 0}
We have the following.

\begin{enumerate}[label = {\rm (\alph*)}]
\item 
Let $m$ and $n$ be nonnegative integers
and $A \in \binom{[m+n]}{n}$. 
For $I\subseteq [m-1]$ and $J\subseteq [n-1]$,  
we have 
\[
{\bf m}_A \left( \dot\chi^I(\nu), \dot\chi^J(\nu) \right) = \dot\chi^{I \shuffle_A J}(\nu).\]

\item 
Let $n\ge 2$ and $1\le k\le n-1$. For $I\subseteq [n-1]$, we have 
\[
\blacktriangle_k(\dot\chi^I(\nu)) =
\dot\chi^{I \cap [k-1]}(\nu) \otimes 
\dot\chi^{ (I \cap [k+1,n-1]) - k}(\nu).\]
\end{enumerate}
\end{lemma}

\begin{proof}
(a) 
Let
\begin{align*}
\dot\chi^I(\nu)_i := 
\begin{cases}
    \mathbbm{1}  & \text{if } i \in I,\\
    \dfrac{\mathbbm{reg-1}}{\nu-1}             & \text{if } i \in [n-1] \setminus I.
\end{cases}
\end{align*}
Since $\chi^I(\nu)(\bm 0) = (\nu-1)^{|I^c|}$, it follows from Proposition~\ref{thm: supercharacter calculation} that   
\[\lb\left(\dot\chi^I(\nu)_i\right)_{i\in [n-1]}\rb= \dot\chi^I(\nu).\]
Let 
\begin{align*}
   {\bf s}_A(\phi, \psi)_i := 
\begin{cases}
    \mathbbm{1}  & \text{if } i \in (I \#_A J)',\\
    \dfrac{\mathbbm{reg-1}}{\nu-1}              & \text{otherwise}
\end{cases}
\end{align*}
for all $i \in [m+n]$.
Combining Definition~\ref{def: I preshuffle J} with~\eqref{def of s(phi,psi)}, 
we derive that 
\[\lb \left( {\bf s}_A(\phi, \psi)_i \right)_{i \in [m+n]} \rb = {\bf s}_A(\phi, \psi).\]
Next, for each $i \in [m+n-1] \setminus c(A)$, we let 
\begin{align*}
    \left({{\bf s}_A(\phi, \psi) \big\downarrow^{Q_{m+n+1}}_{Q_{[m+n-1] \setminus c(A)}}} \right)_i := 
\begin{cases}
    \mathbbm{1}  & \text{if } i \in (I \#_A J) \setminus c(A), \\
    \dfrac{\mathbbm{reg-1}}{\nu-1}             & \text{otherwise}.
\end{cases}
\end{align*}
In view of
$$
\frac{(\mathbbm{reg - 1})(0)}{\nu-1} = \frac{\nu-1}{\nu-1} = 1 \quad \text{and} \quad \mathbbm{1}(0) = 1,
$$
we have 
\[\lb ((({\bf s}_A(\phi, \psi)) \big\downarrow^{Q_{m+n+1}}_{Q_{[m+n-1] \setminus c(A)}})_i)_{i \in [m+n-1] \setminus c(A)} \rb = {\bf s}_A(\phi, \psi) \big\downarrow^{Q_{m+n+1}}_{Q_{[m+n-1] \setminus c(A)}}.\]
Finally, let 
\begin{align*}
    {\bf m}_A(\dot\chi^I(\nu), \dot\chi^J(\nu))_i :=
\begin{cases}
    \mathbbm{1}  & \text{if } i \in c_1(A) \sqcup ((I \#_A J) \setminus c(A)), \\
    \dfrac{\mathbbm{reg-1}}{\nu-1}              & \text{otherwise}
\end{cases}
\end{align*}
for $i \in [m+n-1]$. 
It holds that 
\begin{align*}
    \lb \left({\bf m}_A(\dot\chi^I(\nu), \dot\chi^J(\nu))_i\right)_{i \in [m+n-1]} \rb = 
    \dot\chi^{c_1(A)}(\nu) \otimes_{c(A)} 
\left( {\bf s}_A(\phi, \psi)  \big\downarrow^{Q_{m+n+1}(\nu)}_{Q_{[m+n-1] \setminus c(A)}(\nu)} \right).
\end{align*}
Now the desired result follows from  Definition~\ref{def: product of scf} (a) and~\ref{def: I preshuffle J} (c).

(b) 
Let $\phi = \dot\chi^I(\nu)$, and take
$\phi^k_1 = \lb ({\phi^k_1}_i)_{i \in [k-1]} \rb, \phi^k_2 = \lb ({\phi^k_2}_i)_{i \in [k+1,n-1]} \rb$ as follows:
\begin{align*}
    {\phi^k_1}_i  &= 
\begin{cases}
     \mathbbm{1}  & \text{if } i \in I \cap [k-1], \\
      \frac{\mathbbm{reg-1}}{\nu-1}              & \text{otherwise},
\end{cases}\\
    {\phi^k_2}_i     &=
\begin{cases}
      \mathbbm{1}  & \text{if } i \in I \cap [k+1,n-1], \\
       \frac{\mathbbm{reg-1}}{\nu-1}     &     \text{otherwise}.
\end{cases} 
\end{align*}
It is easy to show that 
\[
\dot\chi^I(\nu) \big\downarrow^{Q_n(\nu)}_{Q_{[1,k-1] \sqcup [k+1,n-1] }(\nu)} = \phi_1^k \otimes_{[1,k-1]} \phi_2^k
\]
and 
\[\phi^k_1 = \dot\chi^{I \cap [k-1]}(\nu) \quad \text{ and } \quad 
\phi^k_2 = \dot\chi^{ (I \cap [k+1,n-1]) - k}(\nu),\]
which proves the assertion.
\end{proof}

\begin{example} \label{first example for multiplication}
Let $m=4$ and $n=3$, and consider the case where 
$I = \{2,3\} \subseteq [m-1]$, $J = \{2\} \subseteq [n-1]$, and $A = \{1,3,4 \}$. 
Then we have
\begin{align*}
    \dot\chi^I(\nu) = 
    \lb \dfrac{\mathbbm{reg-1}}{\nu-1}, \mathbbm{1}, \mathbbm{1} \rb, \quad \text{and} \quad
    \dot\chi^J(\nu) = \lb \dfrac{\mathbbm{reg-1}}{\nu-1}, \mathbbm{1} \rb.
\end{align*}
This implies
\begin{align*}
    &\dot\chi^I(\nu) \otimes_1 \dfrac{\mathbbm{reg-1}}{\nu-1} = 
    \lb \dfrac{\mathbbm{reg-1}}{\nu-1}, \mathbbm{1}, \mathbbm{1}, \dfrac{\mathbbm{reg-1}}{\nu-1} \rb \text { and }\\
    &\dot\chi^J(\nu) \otimes_1 \dfrac{\mathbbm{reg-1}}{\nu-1} = \lb \dfrac{\mathbbm{reg-1}}{\nu-1}, \mathbbm{1}, \dfrac{\mathbbm{reg-1}}{\nu-1} \rb.
\end{align*}
Hence
\begin{align*}
     &{\bf s}_A(\dot\chi^I(\nu), \dot\chi^J(\nu))\\
     &={{\iota}_{A^c}}^\ast \left( \dot\chi^I(\nu) \otimes_1 \dfrac{\mathbbm{reg-1}}{\nu-1} \right) \otimes_{A^c} \left( {{\iota}_{A}}^\ast \left( \dot\chi^J(\nu) \otimes_1 \dfrac{\mathbbm{reg-1}}{\nu-1} \right) \right) \\
     &= 
    \lb \dfrac{\mathbbm{reg - 1}}{\nu-1},  \dfrac{\mathbbm{reg - 1}}{\nu-1}  ,  \mathbbm{1}  , \dfrac{\mathbbm{reg - 1}}{\nu-1}, \mathbbm{1}, \mathbbm{1}, \dfrac{\mathbbm{reg - 1}}{\nu-1} \rb
    \end{align*}
and therefore 
\begin{align*}
    {\bf s}_A(\dot\chi^I(\nu), \dot\chi^J(\nu))  \big\downarrow^{Q_{m+n+1}(\nu)}_{Q_{[m+n-1] \setminus c(A)}(\nu)} 
    = 
    \lb \, \underline{\hspace{0.5cm}},  \underline{\hspace{0.5cm}}  ,  \mathbbm{1}  ,  \underline{\hspace{0.5cm}},  \mathbbm{1},  \mathbbm{1} \rb 
    \in \cf(Q_{[m+n-1] \setminus c(A)}(\nu)).
\end{align*}
Combining the above calculation with the equality 
\[\dot\chi^{c_1(A)}(\nu) = \lb \mathbbm{1}, \dfrac{\mathbbm{reg - 1}}{\nu-1}, \underline{\hspace{0.5cm}},   \mathbbm{1}, \underline{\hspace{0.5cm}}, \underline{\hspace{0.5cm}} \rb \in \cf(Q_{c(A)}(\nu)),\] 
we see that  ${\bf m}_A(\dot\chi^I(\nu), \dot\chi^J(\nu))$ is equal to 
\begin{align*}
     &\dot\chi^{c_1(A)}(\nu) \otimes_{c(A)} ({\bf s}_A(\dot\chi^I(\nu), \dot\chi^J(\nu))  \big\downarrow^{Q_{m+n+1}(\nu)}_{Q_{[m+n-1] \setminus c(A)}(\nu)})\\
    &= \lb \mathbbm{1}, \dfrac{\mathbbm{reg - 1}}{\nu-1}, \mathbbm{1},   \mathbbm{1}, \mathbbm{1}, \mathbbm{1} \rb\\
    &= \dot\chi^{I \shuffle_A J}(\nu).
\end{align*}
\end{example}

\begin{example}
Let $n=5, I = \{1,3,4 \} \subseteq [4]$.
Then 
$$
\dot\chi^I(\nu) = \lb \mathbbm{1}, \dfrac{\mathbbm{reg - 1}}{\nu-1}, \mathbbm{1}, \mathbbm{1} \rb,
$$
thus 
\begin{align*}
  \blacktriangle_0(\dot\chi^I(\nu) ) &= \mathbbm{1}_0 \otimes \lb \mathbbm{1}, \dfrac{\mathbbm{reg - 1}}{\nu-1}, \mathbbm{1},  \mathbbm{1} \rb, & \blacktriangle_3(\dot\chi^I(\nu) ) &= \lb \mathbbm{1}, \dfrac{\mathbbm{reg - 1}}{\nu-1} \rb \otimes \lb \mathbbm{1} \rb,\\
  \blacktriangle_1(\dot\chi^I(\nu) ) &= \mathbbm{1}_1 \otimes \lb \dfrac{\mathbbm{reg - 1}}{\nu-1}, \mathbbm{1}, \mathbbm{1} \rb, &\blacktriangle_4(\dot\chi^I(\nu) ) &= \lb \mathbbm{1}, \dfrac{\mathbbm{reg - 1}}{\nu-1}, \mathbbm{1} \rb \otimes \mathbbm{1}_1,\\
  \blacktriangle_2(\dot\chi^I(\nu) ) &= \lb \mathbbm{1} \rb \otimes \lb \mathbbm{1}, \mathbbm{1} \rb,
   &\blacktriangle_5(\dot\chi^I(\nu) ) &= \lb \mathbbm{1}, \dfrac{\mathbbm{reg - 1}}{\nu-1}, \mathbbm{1}, \mathbbm{1} \rb \otimes \mathbbm{1}_0.
\end{align*}
\end{example}

Lemma~\ref{thm: categorification of QSym 0} implies that the maps ${\bf m}$ and $\blacktriangle$ restrict to the supercharacter function spaces. Thus we have  
\begin{equation}\label{def of product and coproduct}
\begin{aligned}
{\bf m}:\bigoplus_{n \ge 0} \scf(\mathcal{N}_n(\nu))\otimes
\bigoplus_{n \ge 0}\scf(\mathcal{N}_n(\nu)) \to \bigoplus_{n \ge 0}\scf(\mathcal{N}_n(\nu)),\\
\blacktriangle:\bigoplus_{n \ge 0}\scf(\mathcal{N}_n(\nu)) \to \bigoplus_{n \ge 0}\scf(\mathcal{N}_n(\nu)) \otimes \bigoplus_{n \ge 0} \scf(\mathcal{N}_n(\nu)).
\end{aligned}
\end{equation}

For each nonnegative integer $n$, define 
\begin{align*}
    \ch^n_\nu : \scf(\mathcal{N}_n(\nu)) \to \QSym_n
\end{align*}
by the $\mathbb C$-vector space isomorphism given by 
\[\dot\chi^I(\nu)\mapsto L_{\comp(I)} \quad (I \subset [n-1]).\] 
This induces a $\mathbb C$-vector space isomorphism:
\begin{align*}
    \ch_\nu :=\bigoplus_{n \ge 0}\ch^n_\nu: \bigoplus_{n \ge 0} \scf(\mathcal{N}_n(\nu)) \to \QSym
\end{align*}

\begin{lemma}\label{lem: categorification of QSym}
For nonnegative $m$ and $n$, the following hold.

\begin{enumerate}[label = {\rm (\alph*)}]
\item
For $I\subseteq [m-1]$ and $J\subseteq [n-1]$, we have
\begin{align*}
    \ch_\nu ({\bf m}(\dot\chi^I(\nu),\dot\chi^J(\nu)))=L_{\comp(I)} L_{\comp(J)}. 
\end{align*}

\item 
For $I\subseteq [n-1]$, we have
\begin{align*}
    \ch_\nu \otimes \ch_\nu (\blacktriangle(\dot\chi^I(\nu)))=
    \triangle(L_{\comp(I)}),
\end{align*}
where $\triangle$ denotes the coproduct of $\QSym$.
\end{enumerate}
\end{lemma}

\begin{proof}
By Lemma~\eqref{thm: categorification of QSym 0} (a), it suffices to show that 
\begin{equation}\label{eq: L-product, subset parameter}
L_{\comp(I)} L_{\comp(J)} = \sum_{A \in \binom{[m+n]}{n}} L_{\comp(I \shuffle_A J)}.
\end{equation}
Let $w_I \in \SG_m$ be a permutation with $\Des(w_I) = I$ and 
$w_J \in \SG_n$ be a permutation with $\Des(w_J) = J$. 
Theorem~\ref{thm: FQSym to QSym} implies that  
\begin{equation}\label{eq: L-product, set parameter}
L_{\comp(I)} L_{\comp(J)} = \sum_{w \in w_I \shuffle w_J[m]} L_{\comp(\Des(w))}.
\end{equation}
Note that $\textsf{Sh}_{m,n}$ in~\eqref{notation for the set of shuffles} is in bijection with $\binom{[m+n]}{n}$ under the following assignment: 
$$
\sigma \mapsto \{ \sigma^{-1}(m+1), \sigma^{-1}(m+2), \ldots, \sigma^{-1}(m+n) \}.
$$
For $A \in \binom{[m+n]}{n}$, we write  
$\shuffle_A$ for $\shuffle_{\sigma}$, 
where $\sigma$ is the permutation corresponding to $A$ under this bijection.
With this notation, for any word $v$ of length $m$ and $w$ of length $n$, we have the equality:
$$
v \shuffle w = \left\{ v \shuffle_A w \mid A \in \binom{[m+n]}{n} \right\}.
$$
The word $w_I \shuffle_A w_J[m]$ can be obtained from $w_I$ and $w_J[m]$ in the following manner: 
\begin{itemize}
    \item Place the alphabets of $w_J[m]$ in order in the positions occupied by $A$.
    \item Place the alphabets of $w_I$ in order in the positions occupied by $A^c$.
\end{itemize}
Since any alphabet in $w_J[m]$ is larger than the greatest alphabet in $w_I$, it follows that 
$$
c_1(A) \subseteq \Des(w_I \shuffle_A w_J[m]).
$$
Similarly, since any alphabet in $w_I$ is smaller than the smallest alphabet in $w_J[m]$, it follows that 
$$
c_2(A) \subseteq \Des(w_I \shuffle_A w_J[m])^c.
$$
The positions which are not occupied by $c_1(A)$ and $c_2(A)$ depends only on the descent pattern of $w_I \shuffle_A w_J$, 
which implies 
$$
\Des(w_I \shuffle_A w_J[m]) \setminus c(A) = \Des(w_I \shuffle_A w_J) \setminus c(A).
$$
Since 
$$
(I \#_A J) \setminus c(A) = \Des(w_I \shuffle_A w_J)\setminus c(A), 
$$
we have 
\begin{equation*}
\Des(w_I \shuffle_A w_J[m]) = c_1(A) \sqcup ((I \#_A J) \setminus c(A)).
\end{equation*}
It follows that~\eqref{eq: L-product, set parameter} is equal to~\eqref{eq: L-product, subset parameter}.

(b) By Lemma~\eqref{thm: categorification of QSym 0} (b), it suffices to show that 
\begin{align*}
&\triangle L_{\comp(I)} \\
&= L_{\emptyset} \otimes L_{\comp(I)} +  
\sum_{k=1}^{n-1} L_{\comp(I \cap [k-1])} \otimes L_{\comp( (I \cap [k+1,n-1]) - k)} 
+  L_{\comp(I)} \otimes L_{\emptyset}.
\end{align*}
On the other hand, Theorem~\ref{thm: FQSym to QSym} implies that 
$$
\triangle L_{\comp(\Des(w))} = \sum_{k=0}^{n} L_{\comp(\Des( \std(w_1 \ldots w_k) ))} \otimes L_{\comp(\Des( \std(w_{k+1} \ldots w_{n}) ))}.
$$
Let $w_I \in \SG_n$ be a permutation with $\Des(w_I) = I$. 
It is easy to see that 
\begin{align*}
    \Des( \std((w_I)_1 \ldots (w_I)_k) ) &= I \cap [k-1] \subseteq [k-1],\\
    \Des( \std((w_I)_{k+1} \ldots (w_I)_{n}) ) &= (I \cap [k+1, n-1]) - k \subseteq [n-k-1],
\end{align*}
for $k = 1, \ldots, n-1$, as required.
\end{proof}

\begin{example}\label{ex: shuffle of I, J}
Let us revisit 
Example~\ref{first example for multiplication}.
Take $w_I = 1432$, $w_J = 132$.
Then $c_1(A) = \{1,4\}$, $c_2(A) = \{ 2 \}$, and $I \#_A J =\{ 5,6 \} \sqcup \{3 \} $. As a consequence, we have 
$$
w_I \shuffle_A w_J[m] = \{1,4\} \sqcup (\{3,5,6 \} \setminus \{1,2,4\})
= \{1,3,4,5,6 \}.
$$
\end{example}

By Lemma~\ref{lem: categorification of QSym}, we have the following commuting diagrams:
\[\begin{tikzcd}
	\bigoplus_{n \ge 0}\scf(\mathcal{N}_n(\nu))\otimes \bigoplus_{n \ge 0}\scf(\mathcal{N}_n(\nu)) & \bigoplus_{n \ge 0}\scf(\mathcal{N}_n(\nu)) \\
	\QSym \otimes \QSym & \QSym \\
	\bigoplus_{n \ge 0}\scf(\mathcal{N}_n(\nu))\otimes \bigoplus_{n \ge 0}\scf(\mathcal{N}_n(\nu)) & \bigoplus_{n \ge 0}\scf(\mathcal{N}_n(\nu)) \\
	\QSym \otimes \QSym & \QSym
	\arrow[from=1-1, to=1-2, "{\bf m}"]
	\arrow[from=1-2, to=2-2, "\ch_\nu"]
	\arrow[from=2-1, to=2-2, "\cdot"]
	\arrow[from=1-1, to=2-1, "\ch_\nu \otimes \ch_\nu"]
	\arrow[from=3-2, to=3-1, "\blacktriangle"]
	\arrow[from=4-2, to=4-1, "\triangle"]
	\arrow[from=3-2, to=4-2, "\ch_\nu"]
	\arrow[from=3-1, to=4-1, "\ch_\nu \otimes \ch_\nu"]
\end{tikzcd}\]

Now we are ready to state the main result of this section.

\begin{theorem}\label{thm: categorification of QSym}
We have the following.

\begin{enumerate}[label = {\rm (\alph*)}]
\item
$(\bigoplus_{n \ge 0} \scf(\mathcal{N}_n(\nu)), {\bf m}, \blacktriangle)$
has a Hopf algebra structure.

\item
The characteristic map 
\begin{align*}
    \ch_\nu : \bigoplus_{n \ge 0} \scf(\mathcal{N}_n(\nu)) \to \QSym
\end{align*}
is an isomorphism of Hopf algebras.
\end{enumerate}
\end{theorem}

\begin{proof}
By Lemma~\ref{lem: categorification of QSym}, $\ch_\nu$ is a bialgebra isomorphism.
As $\QSym$ is a connected graded Hopf algebra, this proves our assertions.
\end{proof}

\begin{remark}
It should be pointed out that even if $C_\nu$ is replaced by any finite group $G$ of order $q$, all results in this section are still valid.
In this case,  
$Q_S(\nu) = \bigoplus_{s \in S} G_{s}$, where $G_{s} =G$ for all $s \in S$. 
\end{remark}

\subsection{The superclass identifiers of $\bigoplus_{n\ge 0}\scf(\mathcal{N}_n(\nu))$}
\label{sec:The superclass identifiers}
The purpose of this subsection is to 
study the superclass identifiers of $\bigoplus_{n\ge 0}\scf(\mathcal{N}_n(\nu))$.
To be precise, for $I \subseteq [n-1]$ and $J \subseteq [m-1]$, 
we expand 
${\bf m}(\kappa_I(\nu), \kappa_J(\nu) )$ in the basis $\{\kappa_I(\nu): I \subseteq [m+n-1]\}$ for $\scf(\mathcal{N}_{m+n}(\nu))$  
and $\blacktriangle(\kappa_I(\nu))$ in the basis 
for $\bigoplus_{k\in [n-1]}\scf(\mathcal{N}_k(\nu)) \otimes \scf(\mathcal{N}_{n-k}(\nu))$ consisting of tensor products of the superclass identifiers.
This results obtained here will play a crucial role in Section~\ref{Section: new bases for NSym}.

\begin{lemma}\label{lem: vector notation of superclass identifier}
For $I \subseteq [n-1]$, let 
\[
\kappa_I(\nu)_i :=
\begin{cases}
\mathbbm{1} - \frac{1}{\nu} \mathbbm{reg} &\text{ if } i \in I,\\
\frac{1}{\nu} \mathbbm{reg} &\text{ if } i \in [n-1]\setminus I
\end{cases}
\]
for $i \in [n-1]$. Then $\lb ({\kappa_I(\nu)}_i)_{i \in [n-1]} \rb = \kappa_I(\nu)$.
\end{lemma}
\begin{proof}
The assertion follows from 
\begin{align*}
\frac{1}{\nu}\mathbbm{reg}(g)  =
\begin{cases}
1 &\text{ if } g = 0,\\
0 &\text{ otherwise }.
\end{cases}
\end{align*}
\end{proof}

Let us first introduce the expansion of  
${\bf m}(\kappa_I(\nu), \kappa_J(\nu) )$ in the basis 
consisting of the superclass identifiers.

\begin{proposition}\label{thm: product formula of superclass identifier}
Let $I \subseteq [m-1], J \subseteq [n-1]$. Then 
\[
{\bf m}(\kappa_I(\nu), \kappa_J(\nu) ) = \sum_{K \subseteq [m+n-1]} d_K\,\kappa_K(\nu),
\]
where 
\[d_K=\sum_{\substack{A \in \binom{[m+n]}{n}:\\ (I \#_A J) \cap c(A) = \emptyset \\ I \#_A J \subseteq K \subseteq (I \#_A J ) \cup c(A) }} \left(\dfrac{1}{1-\nu}\right)^{|K \cap c_2(A)|} .\]
\end{proposition}

\begin{proof}
Observe the following equalities:
\begin{align}\label{eq: 1 decomposition}
    \mathbbm{1} = \left(\mathbbm{1} - \frac{1}{\nu} \mathbbm{reg}\right) + \frac{1}{\nu} \mathbbm{reg},
\end{align}
\begin{align}\label{eq: reg-1/reg-1 decomposition}
    \frac{\mathbbm{reg} - \mathbbm{1}}{\nu-1} = \left(\frac{-1}{\nu-1}\right)  \left(\mathbbm{1} - \frac{1}{\nu} \mathbbm{reg}\right) + \frac{1}{\nu} \mathbbm{reg}.
\end{align}
For $A \subseteq [m+n]$, 
set 
\[
z_A:= 
\begin{cases}
\max A & \text{ if } m+n \notin A,\\
\max \,\,[m+n]\setminus A & \text{ if } m+n \in A.
\end{cases}
\]
If we let
\[
\phi_i : = 
\begin{cases}
\mathbbm{1} - \frac{1}{\nu} \mathbbm{reg} &\text{ if } i \in (I \#_A J)', \\
\frac{\mathbbm{reg} - \mathbbm{1}}{\nu-1} &\text{ if } i = z_A \text{ or } i= m+n,\\
\frac{1}{\nu} \mathbbm{reg} &\text{ otherwise,}
\end{cases}
\]
then 
\[
\lb (\phi_i)_{i \in [m+n]} \rb =
{\bf s}_A(\kappa_I(\nu), \kappa_J(\nu)).
\]
Applying~\eqref{eq: reg-1/reg-1 decomposition} 
to the $\phi_{z_A}$ in this vector notation and then 
using Lemma~\ref{lem: vector notation of superclass identifier}, 
one can simplify 
\begin{equation*}
({\bf s}_A(\kappa_I(\nu), \kappa_J(\nu))) \big\downarrow^{Q_{m+n+1}(\nu)}_{Q_{[m+n-1] \setminus c(A)}(\nu)}
\end{equation*}
in the following simple form:
\begin{equation}\label{tensor product for multiplication of identifier}
\left( \left(\frac{-1}{\nu-1}\right)   \kappa_{(I \#_A J) \sqcup \{z_A\}}(\nu) + \kappa_{I \#_A J}(\nu) \right) \Big\downarrow^{Q_{m+n+1}(\nu)}_{Q_{[m+n-1] \setminus c(A)}(\nu)}
\end{equation}
On the other hand, one can easily see that for any subsets $S,T \subseteq [n-1]$, 
\begin{align}\label{eq: restriction of superclass identifier}
\kappa_S(\nu) \big\downarrow^{Q_n(\nu)}_{Q_T(\nu)} = 
\begin{cases}
\kappa_{S}(\nu) &\text{ if } S \subseteq T,\\
0 &\text{ otherwise }.
\end{cases}
\end{align}
Since $z_A \in c(A)$, from \eqref{eq: restriction of superclass identifier} it follows that 
\[
\eqref{tensor product for multiplication of identifier}=
\begin{cases}
\kappa_{I \#_A J}(\nu) &\text{ if } I \#_A J \subseteq [m+n-1] \setminus c(A),\\
0 &\text{ otherwise }.
\end{cases}
\]
Finally, using~\eqref{eq: 1 decomposition} and~\eqref{eq: reg-1/reg-1 decomposition}, we derive that 
\begin{align*}
&{\bf m}_A(\kappa_I(\nu), \kappa_J(\nu) )\\ 
&=
\begin{cases}
\dot\chi^{c_1(A)}(\nu) \otimes_{c(A)} \kappa_{I \#_A J}(\nu) &\text{ if } I \#_A J \subseteq [m+n-1] \setminus c(A),\\
0 &\text{ otherwise }
\end{cases}\\
&=
\begin{cases}
\sum_{I \#_A J \subseteq K \subseteq (I \#_A J) \sqcup c(A) } \left(\dfrac{1}{1-\nu}\right)^{|K \cap c_2(A)|}  \kappa_{K}(\nu) &\text{ if } I \#_A J \subseteq [m+n-1] \setminus c(A),\\
0 &\text{ otherwise }.
\end{cases}
\end{align*}
This completes the proof.
\end{proof}

\begin{example}
Let $m = 2, n = 3$ and $I = \{1\}, J = \{2\}$.
The following table shows all the statistics required to calculate 
${\bf m}(\kappa_I(\nu), \kappa_J(\nu))$.

%column 너비 조정, centering
\newcolumntype{C}[1]{>{\centering\arraybackslash}p{#1}}

%horizontal 굵은 줄
\makeatletter
\def\hlinewd#1{%
\noalign{\ifnum0=`}\fi\hrule \@height #1 \futurelet
\reserved@a\@xhline}
\makeatother

%vertical 굵은 줄
\newcolumntype{?}{!{\vrule width 1pt}}

%font크기 조절
%\tiny, \scriptsize, \footnotesize, \small, \normalsize, \large, \Large, \LARGE, \huge, and \Huge.

%table 시작
\begin{table}[ht]
\small{
\begin{tabular}
{|C{1.5cm}?C{1.5cm}|C{1.5cm}|C{1.5cm}|C{1.5cm}|C{2.4cm}|C{2.4cm}|}
\hline
& & & & & & \\[-1em]
 \textbf{$A$}  & \textbf{$c_1(A)$}       & \textbf{$c_2(A)$}        &  \textbf{$c(A)$}  &  \textbf{$I \#_A J$}  &  \textbf{$(I \#_A J) \cap c(A)$}  &  \textbf{$(I \#_A J) \cup c(A)$}  \\[-1em]
& & & & & &  \\ \hlinewd{1pt}
 \{1,2,3\}   & \{3\}                    & $\emptyset$ & \{3\}           &  \{2,4\}              & $\emptyset$                                           & \{2,3,4\}                      \\ \hline
\{1,2,4\}    & \{2,4\}                 & \{3\}                    & \{2,3,4\}       & \{2,3\}             & $\{2,3\}$                                           & unnecessary                                \\ \hline
\{1,2,5\}    & \{2\}                    & \{4\}                    & \{2,4\}         & \{2,3\}             & $\{2\}$                                           & unnecessary                                \\ \hline
\{1,3,4\}    & \{1,4\}                  & \{2\}                    & \{1,2,4\}       & \{2,3\}             & $\{2\}$                                           & unnecessary                                \\ \hline
\{1,3,5\}    & \{1,3\}                  & \{4\}                    & \{1,3,4\}       & \{2,3\}             & $\{3\}$                                           &  unnecessary                               \\ \hline
\{1,4,5\}    & \{1\}                    & \{3\}                    & \{1,3\}         & \{2,4\}             & $\emptyset$                                           & \{1,2,3,4\}                     \\ \hline
\{2,3,4\}    & \{4\}                    & \{1\}                    & \{1,4\}         & \{1,3\}             & $\{1\}$                                           & unnecessary                                \\ \hline
\{2,3,5\}    & \{3\}                    & \{1,4\}                  & \{1,3,4\}       & \{1,3\}             & $\{1,3\}$                                           &  unnecessary                               \\ \hline
\{2,4,5\}    & \{2\}                    & \{1,3\}                  & \{1,2,3\}       & \{1,4\}             & $\{1\}$                                           & unnecessary                                \\ \hline
\{3,4,5\}    & $\emptyset$ & \{2\}                    & \{2\}           & \{1,4\}             & $\emptyset$                                           & \{1,2,4\}                       \\ \hline
\end{tabular}}
\end{table}

\noindent
From this, we have
\begin{align*}
{\bf m}(\kappa_I(\nu), \kappa_J(\nu) ) &= \left(\dfrac{1}{1-\nu}\right)^{0}  \kappa_{\{1,4 \}}(\nu) + \left\{\left(\dfrac{1}{1-\nu}\right)^{0}+\left(\dfrac{1}{1-\nu}\right)^{0}\right\}  \kappa_{\{2,4 \}}(\nu)\\
&+ \left\{\left(\dfrac{1}{1-\nu}\right)^{0}+\left(\dfrac{1}{1-\nu}\right)^{1}\right\}  \kappa_{\{1,2,4 \}}(\nu)\\
&+ \left\{\left(\dfrac{1}{1-\nu}\right)^{0}+\left(\dfrac{1}{1-\nu}\right)^{1}\right\}  \kappa_{\{2,3,4 \}}(\nu)
+\left(\dfrac{1}{1-\nu}\right)^{1}  \kappa_{\{1,2,3,4 \}}(\nu)\\
&= \kappa_{\{1,4 \}}(\nu) + 2 \kappa_{\{1,4 \}}(\nu) + \left(\dfrac{2-\nu}{1-\nu}\right)  \kappa_{\{1,2,4 \}}(\nu) + \left(\dfrac{2-\nu}{1-\nu}\right)  \kappa_{\{1,3,4 \}}(\nu)\\
&+ \left(\dfrac{1}{1-\nu}\right)  \kappa_{\{1,2,3,4 \}}(\nu).
\end{align*}

\end{example}

For two compositions $\alpha = (\alpha_1, \alpha_2, \ldots, \alpha_l)$ and $\beta = (\beta_1, \beta_2, \ldots, \beta_k )$, the \emph{near-concatenation} $\alpha \odot \beta$ of $\alpha$ and $\beta$ is defined by the composition  
\[
\alpha \odot \beta := (\alpha_1, \alpha_2, \ldots, \alpha_l + \beta_1, \beta_2, \ldots, \beta_k).
\]
The next result concerns the expansion of $\blacktriangle(\kappa_I(\nu))$.
\begin{proposition}\label{thm: coproduct formula of superclass identifier}
Let $n$ be a nonnegative integer.
For $\gamma \in \Comp_n$, we have
\[
\blacktriangle(\kappa_{\set(\gamma)}(\nu)) = \sum_{\alpha \odot \beta = \gamma} \kappa_{\set(\alpha)}(\nu) \otimes \kappa_{\set(\beta)}(\nu).
\]
\end{proposition}

\begin{proof}
Let 
$A := \{(\alpha, \beta) \mid \alpha \cdot \beta = \gamma \}$ and 
$B := \{(\alpha, \beta) \mid \alpha \odot \beta = \gamma \}$.
Then
$A \cap B = \{ (\emptyset, \gamma),(\gamma, \emptyset)  \}$,
where $\emptyset$ denotes the empty composition in $\Comp_0$.
It is well known that 
\[
\triangle(L_{\gamma}) = \sum_{\alpha \cdot \beta = \gamma \text{ or } \alpha \odot \beta = \gamma} L_{\alpha} \otimes L_{\beta}
\]
(see~\cite[Proposition 5.2.15]{GR20}).
Therefore, by Theorem~\ref{thm: categorification of QSym}, we have
\begin{equation}\label{coproduct of chi(nu)}
\begin{aligned}
\blacktriangle(\dot\chi^{\set(\gamma)}(\nu)) 
&= \sum_{\alpha \cdot \beta = \gamma \text{ or } \alpha \odot \beta = \gamma} \dot\chi^{\set(\alpha)}(\nu) \otimes \dot\chi^{\set(\beta)}(\nu)\\
&= \sum_{(\alpha, \beta) \in A \cup B} \dot\chi^{\set(\alpha)}(\nu) \otimes \dot\chi^{\set(\beta)}(\nu).
\end{aligned}
\end{equation}
For simplicity, set $I := \set(\gamma)$.
Combining Lemma~\ref{thm: categorification of QSym 0} (b) with~\eqref{coproduct of chi(nu)}, it can be easily seen that $A \cup B$ is equal to  
\[
\{( \comp(I \cap [1,k-1]), \comp((I \cap [k+1, n-1]) - k) ) \mid 1 \le k \le n-1 \} \sqcup \{ (\emptyset, \gamma),(\gamma, \emptyset)  \}.
\]
For $1 \le k \le n-1$, it holds that  
\begin{equation}\label{eq: concatenation, near concatenation in set}
    \comp(I)
    =\begin{cases}
    \comp(I \cap [1,k-1]) \odot \comp((I \cap [k+1, n-1]) - k) & \text{if }k \notin I,\\
    \comp(I \cap [1,k-1]) \cdot \comp((I \cap [k+1, n-1]) - k) & \text{if }k \in I,
    \end{cases}
\end{equation}
and which implies 
\[
B = \{( \comp(I \cap [1,k-1]), \comp((I \cap [k+1, n-1]) - k) ) \mid k \notin I \} \sqcup \{ (\emptyset, \gamma),(\gamma, \emptyset)  \}.
\]
On the other hand, by~\eqref{eq: restriction of superclass identifier}, 
we have that 
\begin{equation}\label{eq: kappa coproduct}
\blacktriangle_k(\kappa_I(\nu)) = 
\begin{cases}
\kappa_{I \cap [1,k-1]}(\nu) \otimes (\iota^\ast_{[k+1,n-1]})^{-1} (\kappa_{I \cap [k+1, n-1]}(\nu)) &\text{ if } k \notin I,\\
0 &\text{ if } k \in I
\end{cases}
\end{equation}
for $1 \le k \le n-1$.
If $k=0$ or $n$, then
$\blacktriangle_0(\phi)
=\mathbbm{1}_0 \otimes \phi \text{ and } \blacktriangle_n(\phi)
=\phi \otimes \mathbbm{1}_0,
$
so 
\begin{align*}
&\blacktriangle_0(\kappa_{\set(\gamma)}(\nu))
=\kappa_{\set(\emptyset)}(\nu) \otimes \kappa_{\set(\gamma)}(\nu) \text{ and }\\ &\blacktriangle_n(\kappa_{\set(\gamma)}(\nu))
=\kappa_{\set(\gamma)}(\nu) \otimes \kappa_{\set(\emptyset)}(\nu).
\end{align*}
Putting these together, we conclude that 
\[
\blacktriangle(\kappa_I(\nu)) = \sum_{k=0}^{n}\blacktriangle_k(\kappa_I(\nu)) = \sum_{(\alpha, \beta) \in B} \kappa_{\set(\alpha)}(\nu) \otimes \kappa_{\set(\beta)}(\nu).
\]
\end{proof}

\begin{example}
Let $\gamma = (1,3,2) = \comp(\{1,4\})$. 
All the possible ways to write $\gamma$ as $\alpha \odot \beta$ are 
\[\emptyset \odot (1,3,2) = (1,1) \odot (2,2) = (1,2) \odot (1,2) = (1,3,1) \odot (1) = (1,3,2) \odot \emptyset.\]
Therefore 
$\blacktriangle(\kappa_{\{1,4\}})$ is equal to 
 \begin{align*}
  \mathbbm{1}_0 \otimes \kappa_{\{1,4\}}(\nu) +
    \kappa_{\{1\}}(\nu) \otimes \kappa_{\{2\}}(\nu) 
    +\kappa_{\{1\}}(\nu) \otimes \kappa_{\{2\}}(\nu) 
  +\kappa_{\{1,4\}}(\nu) \otimes \mathbbm{1}_1 
    +\kappa_{\{1,4\}}(\nu) \otimes \mathbbm{1}_0,
\end{align*}
where 
\begin{align*}
    &\mathbbm{1}_0 \otimes \kappa_{\{1,4\}}(\nu) \in  \scf(\mathcal{N}_0(\nu)) \otimes \scf(\mathcal{N}_6(\nu)), \quad 
    \kappa_{\{1\}}(\nu) \otimes \kappa_{\{2\}}(\nu) \in  \scf(\mathcal{N}_2(\nu)) \otimes \scf(\mathcal{N}_4(\nu)),\\
    &\kappa_{\{1\}}(\nu) \otimes \kappa_{\{2\}}(\nu) \in \scf(\mathcal{N}_3(\nu)) \otimes \scf(\mathcal{N}_3(\nu)), \quad 
    \kappa_{\{1,4\}}(\nu) \otimes \mathbbm{1}_1 \in  \scf(\mathcal{N}_5(\nu)) \otimes \scf(\mathcal{N}_1(\nu)),\\
    &\kappa_{\{1,4\}}(\nu) \otimes \mathbbm{1}_0 \in  \scf(\mathcal{N}_6(\nu)) \otimes \scf(\mathcal{N}_0(\nu)).
\end{align*}

\end{example}

\section{A new basis for $\NSym_{\mathbb{C}(q,t)}$ and the structure constants}\label{Section: new bases for NSym}
Although defined over the base field $\mathbb C$.
the Hopf algebras in Section~\ref{Hopf algebras in consideration}
can be defined over any nonzero field (see \cite{GR20}).
The main object of the present section is $\NSym_{\mathbb{C}(q,t)}$, the Hopf algebra of noncommutative symmetric functions defined over $\mathbb{C}(\mathfrak q,\mathfrak t)$,
where $q$ and $t$ are commuting variables.
We introduce and investigate a new basis $\{\mathcal{B}(q,t)_{\alpha}\mid \alpha \in \Comp\}$ for $\NSym_{\mathbb{C}(q,t)}$.
We also consider its variant $\{\mathcal{\widehat{B}}(q,t)_{\alpha}\mid \alpha \in \Comp\}$ obtained by re-parametrizing the indices.
Particular emphasis will be placed on the structure constants for these bases and Theorem~\ref{thm: categorification of QSym} will play an important role in performing this.
As before, $n$ is assumed to be any nonnegative integer 
throughout this section.

\subsection{A new basis $\{\mathcal{B}(q,t)_{\alpha}\mid \alpha \in \Comp\}$ for $\NSym_{\mathbb C(q,t)}$}\label{Subsection: qt basis B for NSym}

We begin with the definition of $\mathcal{B}(q,t)_{\alpha}$.

\begin{definition}\label{defn of B(q,t)}
For $I \subseteq [n-1]$, define
\[
\mathcal{B}(q,t)_{\comp(I)} := \sum_{J: I \cup J = [n-1]} q^{|I \setminus J|} \  t^{|I \cap J|} \  H_{\comp(J)} \in \NSym_{\mathbb{C}(q,t)}.
\]
\end{definition}

Set 
\[\alpha^c:=\comp([n-1]\setminus \set(\alpha))\]
for all $\alpha \in \Comp_n$.
Then, choose a linear extension $\preceq^c_{{\rm lin.ext.}}$ of the partial order $\preceq^c$ on $\{\mathcal{B}(q,t)_{\alpha} | \alpha\in \Comp_n \}$, where 
$$\mathcal{B}(q,t)_{\alpha} \preceq^c \mathcal{B}(q,t)_{\beta} \text{ if and only if } \mathcal{B}(q,t)_{\alpha^c} \preceq \mathcal{B}(q,t)_{\beta}^c.$$
And, we also choose a linear extension $\preceq_{{\rm lin.ext.}}$ of the partial order $\preceq$ on $H_{\alpha}$. 
Since $I \cup J = [n-1]$ if and only if $I^c \subseteq J$ if and only if $\comp(J) \preceq \comp(I)^c$,
the transition matrix from the ordered basis $(\{\mathcal{B}(q,t)_{\alpha}\}, \preceq^c_{{\rm lin.ext.}})$ to the ordered basis $(\{H_{\alpha}\}, \preceq_{{\rm lin.ext.}})$ is triangular with non-zero diagonals, 
Therefore $\{ \mathcal{B}(q,t)_{\alpha} \mid \alpha \in \Comp \}$ is a basis for $\NSym_{\mathbb{C}(q,t)}$.
It is quite interesting to note that this basis recovers 
well known bases for $\NSym$ such as $\{\Lambda_{\alpha}\}, \{H_{\alpha}\}, \{E^*_{\alpha}\}$ via suitable specializations of $q$ and $t$.

\begin{proposition}\label{prop: specializations of B(q,t)}
For each $\alpha \in \Comp$, we have
\begin{enumerate}[label = {\rm (\alph*)}]
    \item $\mathcal{B}(1,0)_{\alpha} = H_{\alpha^c}$,
    \item $\mathcal{B}(-1,1)_{\alpha} = \Lambda_{\alpha^c}$,
    \item $\mathcal{B}(1,-1)_{\alpha} = E^*_{\alpha^c}$.
\end{enumerate}
\end{proposition}

\begin{proof}
(a) The assertion follows from the identity 
$
1^{|I \setminus J|} \ 0^{|I \cap J|} = 
\begin{cases}
1 &\text{ if } I \cap J = \emptyset,\\
0 &\text{ otherwise. }
\end{cases}
$

(b) 
Recall that the noncommutative elementary symmetric function $\Lambda_{\alpha}$ is given by
$$
\Lambda_{\alpha} = \sum_{\beta \preceq \alpha} (-1)^{n-\ell(\beta)} H_{\beta}.
$$
Therefore the assertion follows from the calculation below: 
\begin{align*}
\Lambda_{\comp(I^c)} 
&= \sum_{I^c \subseteq J} (-1)^{n-(|J|+1)} H_{\comp(J)}\\
&= \sum_{I^c \subseteq J} (-1)^{|J^c|} H_{\comp(J)}\\
&= \sum_{J: I \cup J = [n-1]} (-1)^{|I \setminus J|} \  1^{|I \cap J|} \  H_{\comp(J)}.
\end{align*}

(c) 
Recall that the dual essential quasisymmetric function $E^*_{\alpha}$ is given by
$$
E^*_{\alpha} = \sum_{\beta \preceq \alpha} (-1)^{\ell(\beta)-\ell(\alpha)} H_{\beta}.
$$
Therefore, 
$$
E^*_{\comp(I^c)} = \sum_{I^c \subseteq J} (-1)^{(|J|+1)-(|I^c|+1)} H_{\comp(J)}
= \sum_{I^c \subseteq J} (-1)^{|J|-|I^c|} H_{\comp(J)}.
$$
This completes the proof 
since $I^c \subseteq J$ implies $|J|-|I^c| = |I \cap J|$.
\end{proof}

\subsection{The structure constants of $\NSym_{\mathbb C(q,t)}$ for  $\{\mathcal{B}(q,t)_{\alpha}\mid \alpha \in \Comp\}$}\label{subsection: Structure constants of N for the basis B}
As in Section~\ref{Section: categorification of QSym}, $\nu$ denotes any positive integer $>1$. 
For each $I \subseteq [n-1]$, let
\[
\Pi(\nu)_{\comp(I)} := \ch_\nu^n\left(\frac{\kappa_I(\nu)}{(\nu-1)^{|I|}}\right).
\]
Since $\{\kappa_I(\nu): I \subseteq [n-1]\}$ is a basis for
$\scf(\mathcal{N}_n(\nu))$ and $\ch_\nu^n$ is an isomorphism,  
$\{\Pi(\nu)_{\alpha}: \alpha\in \Comp_n\}$ is a basis for  $\QSym_n$.
The following lemma shows the changes of basis between   
$\{L_{\alpha}: \alpha\in \Comp_n\}$ and 
$\{\Pi(\nu)_{\alpha}: \alpha\in \Comp_n\}$.

\begin{lemma}\label{thm: transition between Pi and L} 
Let $I,J \subseteq [n-1]$. Then  
\begin{enumerate}[label = {\rm (\alph*)}]
\item
$
L_{\comp(I)} = \displaystyle \sum_{J} (-1)^{|J \setminus I|}  (\nu-1)^{|I \cap J|}  \Pi(\nu)_{\comp(J)} 
$ and 
\item
$
\Pi(\nu)_{\comp(J)} = \displaystyle\sum_{I} \dfrac{1}{\nu^{n-1}}  (-1)^{|J \setminus I|}  (\nu-1)^{|(I \cup J)^c|}  L_{\comp(I)}.
$
\end{enumerate}
\end{lemma}

\begin{proof}
For $I,J \subseteq [n-1]$, let $\chi^I_J(\nu)$ be the coefficient of $\kappa_J(\nu)$ in $\chi^I(\nu)$, that is, 
\begin{align}\label{eq: supercharacter to superclass identifier}
\chi^I(\nu) = \sum_{J} \chi^I_J(\nu) \kappa_J(\nu).
\end{align}
Proposition~\ref{thm: supercharacter calculation} implies that  
$\chi^I_J(\nu) = (-1)^{|J \setminus I|}  (\nu-1)^{|(I \cup J)^c|}.$  
Let $\langle \cdot, \cdot \rangle$ be the Hall-inner product on $\cf(Q_n(\nu))$, that is,
\[
\langle \phi, \psi \rangle = \frac{1}{|Q_n(\nu)|} \sum_{\bm{g} \in Q_n(\nu)} \phi(\bm{g}) \psi(\bm{g}^{-1}) \quad (\phi, \psi \in \cf(Q_n(\nu))).
\]
By Proposition~\ref{thm: supercharacter calculation} and Lemma~\ref{lem: vector notation of superclass identifier}, we see that  
\begin{align*}
\langle \chi^I(\nu), \chi^I(\nu)\rangle=(\nu-1)^{|I^c|} \quad \text{ and }
\quad \langle \kappa_I(\nu), \kappa_I(\nu)\rangle =\dfrac {\nu^{n-1}}{(\nu-1)^{|I|}}.
\end{align*}
Letting  
$$
{\overline {\chi}}^I(\nu):=\sqrt{\dfrac{1}{(\nu-1)^{|I^c|}}} \,\, \chi^I(\nu) \quad \text{and} \quad
{\overline \kappa}^I(\nu):=\sqrt{\dfrac{(\nu-1)^{|I|}}{\nu^{n-1}}} \,\, \kappa_I(\nu),
$$
one can see that $\{{\overline {\chi}}^I(\nu): I \subseteq [n-1]\}$
and $\{{\overline \kappa}^I(\nu): I \subseteq [n-1]\}$ are 
orthonormal bases for $\scf(\mathcal{N}_n(\nu))$.
Using this notation, 
we can rewrite~\eqref{eq: supercharacter to superclass identifier} as 
\begin{align}\label{eq: supercharacter to superclass identifier1}
{\overline {\chi}}^I(\nu) = \sum_{J} {\overline {\chi}}^I_J(\nu) {\overline \kappa}_J(\nu)
\end{align}
with
$${\overline {\chi}}^I_J(\nu):=\frac{\chi^I_J(\nu)}{\sqrt{(\nu-1)^{|I^c|}}} \ \sqrt{\frac{(\nu-1)^{|J|}}{\nu^{n-1}}}.$$
Since the matrix $\left({\overline \chi}^I_J(\nu)\right)_{I,J}$ is unitary, 
we have 
\begin{align}\label{eq: supercharacter to superclass identifier2}
{\overline \kappa}^J(\nu) = \sum_{I} {\overline {\chi}}^I_J(\nu) {\overline {\chi}}_I(\nu).
\end{align}
Now, our assertions can be obtained by taking $\ch_\nu$ on both sides~\eqref{eq: supercharacter to superclass identifier1} and 
~\eqref{eq: supercharacter to superclass identifier2}.
\end{proof}

The following lemma shows the change of basis for $\{M_{\alpha}\}$ and $\{\Pi(\nu)_{\alpha}\}$, where 
$\{M_{\alpha}\}$ is the basis of the monomial quasisymmetric functions.

\begin{lemma}\label{lem: transition between Pi and M}
Let $I,J \subseteq [n-1]$. Then
\begin{enumerate}[label = {\rm (\alph*)}]
\item
$M_{\comp(I)} = \displaystyle\sum_{J: I \cup J = [n-1]} (-\nu)^{|J \setminus I|}  (\nu-1)^{|I \cap J|}  \Pi(\nu)_{\comp(J)}$ and 
\item
$
\Pi(\nu)_{\comp(J)} = \left(\dfrac{1}{1-q}\right)^{|J|} \displaystyle\sum_{I: I \cap J = \emptyset} \left(\frac{q-1}{q}\right)^{(n-1) - |I|}  M_{\comp(I)}.
$
\end{enumerate}
\end{lemma}

\begin{proof}
(a) In view of Lemma~\ref{thm: transition between Pi and L} (a), we have that   
\begin{align*}
    M_{\comp(I)} &= \sum_{K: I \subseteq K} (-1)^{|K \setminus I|}  L_{\comp(K)}\\
    &= \sum_{K: I \subseteq K} (-1)^{|K \setminus I|}  \left( \sum_{J} (-1)^{|J \setminus K|}  (\nu-1)^{|K \cap J|}  \Pi(\nu)_{\comp(J)} \right)\\
    &= \sum_{J} \left( \sum_{K: I \subseteq K} (-1)^{|K \setminus I|}   (-1)^{|J \setminus K|}  (\nu-1)^{|K \cap J|} \right)  \Pi(\nu)_{\comp(J)}.
\end{align*}
For simplicity, we use the following abbreviations: 
\[
V:= ([n-1] \setminus (I \cup J)) \cap K \quad \text{and} \quad
W:= (J \setminus I) \cap K.
\]
Then $K = I \sqcup V \sqcup W$, and therefore our assertion follows from the calculation below: 
\begin{align*}
    &\sum_{K: I \subseteq K} (-1)^{|K \setminus I|}   (-1)^{|J \setminus K|}  (\nu-1)^{|K \cap J|} \\
    &=  \sum_{\substack{V \subseteq [n-1] \setminus (I \cup J) \\W \subseteq J \setminus I }} (-1)^{|V|+|W|}   (-1)^{|J \setminus I|- |W|}  (\nu-1)^{|I \cap J|+|W|}\\
    &= (-1)^{|J \setminus I|}  (\nu-1)^{|I \cap J|}  \sum_{\substack{V \subseteq [n-1] \setminus (I \cup J) \\W \subseteq J \setminus I }} (-1)^{|V|}  (\nu-1)^{|W|}\\
    &= (-1)^{|J \setminus I|}  (\nu-1)^{|I \cap J|}  0^{|[n-1] \setminus (I \cup J)|} \ \nu^{|J \setminus I|}.
\end{align*}
The last equality can be derived by applying 
the binomial expansion formula.

(b) Combining Lemma~\ref{thm: transition between Pi and L} (b) with 
$L_{\comp(K)} = \sum_{I: K \subseteq I} M_{\comp(I)}$, we derive that 
\begin{align*}
    \Pi(\nu)_{\comp(J)}
    &= \sum_{K} \frac{1}{\nu^{n-1}}  (-1)^{|J \setminus K|}  (\nu-1)^{|(K \cup J)^c|}  \left(
    \sum_{I: K \subseteq I} M_{\comp(I)}
    \right)\\
    &= \frac{1}{\nu^{n-1}}  \sum_{I} 
    \left(
    \sum_{K: K \subseteq I}
    (-1)^{|J \setminus K|}  (\nu-1)^{|(K \cup J)^c|} 
    \right) 
    M_{\comp(I)}.
\end{align*}
We use the following abbreviations: 
\[
X:= (I \cap J) \cap K \quad \text{and} \quad Y:=(I \setminus J) \cap K.
\]
Then it holds that $K = V \sqcup Y$ whenever $K \subseteq I$. 
Now, our assertion follows from the calculation below:
\begin{align*}
    &\sum_{K: K \subseteq I}
    (-1)^{|J \setminus K|}  (\nu-1)^{|(K \cup J)^c|} \\
    &=
    \sum_{\substack{X \subseteq I \cap J \\Y \subseteq I \setminus J}}
    (-1)^{|J| - |X|}  (\nu-1)^{(n-1)-|J|-|Y|}\\
    &= (-1)^{|J|}  0^{|I \cap J|}  (\nu-1)^{(n-1) - |J|}  \left(1+\frac{1}{\nu-1}\right)^{|I \setminus J|}.
\end{align*}
\end{proof}

\begin{remark}
Let $M$ be the $2^{n-1} \times 2^{n-1}$ matrix whose columns and rows are indexed by subsets of $[n-1]$ and with entries 
\[
M_{I,J} = 
\begin{cases}
q^{|J \setminus I|} t^{|J \cap I|} &\text{ if }  I \cup J = [n-1],\\
0 &\text{ otherwise.}
\end{cases}
\]
In fact, it is the transition matrix from $\mathcal{B}(q,t)_{\comp(I)}$ to $H_{\comp(J)}$.
Lemma~\ref{lem: transition between Pi and M} (a) implies that $\Pi(\nu)_{\comp(I)} = \mathcal{B}(-\nu,\nu-1)_{\comp(I)}^*$ since $\{ M_{\alpha} \}$ is the dual basis of $\{H_{\alpha}\}$. 
Hence Lemma~\ref{lem: transition between Pi and M} (b) implicitly implies  that the inverse matrix $N$ of $M$ is given by
\[
N_{I,J} =
\begin{cases}
 q^{|J| - (n-1)} (-t)^{(n-1) - |I| - |J|} &\text{ if }  I \cap J = \emptyset,\\
0 &\text{ otherwise.}
\end{cases}
\]
This can be verified by a direct calculation.
Consequently, letting $\{ \mathcal{B}(q,t)_{\alpha}^* \mid \alpha \in \Comp\}$ be the dual basis of $\{ \mathcal{B}(q,t)_{\alpha} \mid \alpha \in \Comp\}$, 
we have the following expansion:
\[
\mathcal{B}(q,t)_{\comp(I)}^* = \sum_{J: I \cap J = \emptyset}
 q^{|J| - (n-1)} (-t)^{(n-1) - |I| - |J|} M_{\comp(J)}.
\]
\end{remark}

Next, we study $\Pi(q)^*$, the dual basis of $\Pi(q)$, which plays an important role in proving Theorem~\ref{thm: str consts for B}.
\begin{lemma}\label{lem: str consts for Pi}
Let $m$, $n$, and $k$ be nonnegative integers.

\begin{enumerate}[label = {\rm (\alph*)}]
    \item Let $K \subseteq [k-1]$. Then 
    \[
    \triangle \Pi(\nu)_{\comp(K)}^* =
    \sum_{\substack{ m+n =k\\ I \subseteq [m-1] \\ J \subseteq [n-1] }} C^K_{I,J}(\nu)  (\Pi(\nu)_{\comp(I)}^* \otimes \Pi(\nu)_{\comp(J)}^*),
    \] 
    where
    \[
    C^K_{I,J}(\nu) = \frac{1}{(\nu-1)^{|I|+|J|}} \sum_{\substack{A \in \binom{[m+n]}{n}:\\ (I \#_A J) \cap c(A) = \emptyset \\ I \#_A J \subseteq K \subseteq (I \#_A J ) \cup c(A) }} (-1)^{|K \cap c_2(A)|}  (\nu-1)^{|K \setminus c_2(A)|}.
    \]
    \item For $\alpha \in \Comp_m$ and $\beta \in \Comp_n$,\\
\[
\Pi(\nu)_{\alpha}^*  \Pi(\nu)_{\beta}^* = \Pi(\nu)_{\alpha \odot \beta}^*.
\]
\end{enumerate}
\end{lemma}
\begin{proof}
Recall that the structure constant for the product(coproduct) of $\Pi^*$ is equal to the structure constant for the coproduct(product) of $\Pi$, that is, 
\[
\Pi(\nu)_{\alpha}^*   \Pi(\nu)_{\beta}^* = \sum_{\gamma} C^{\gamma}_{\alpha, \beta} \Pi_{\gamma}^* 
\,\, \Longleftrightarrow \,\, 
\triangle \Pi(\nu)_{\gamma} = \sum_{\alpha, \beta} C^{\gamma}_{\alpha, \beta} (\Pi(\nu)_{\alpha} \otimes \Pi(\nu)_{\beta}).
\]
Also, we recall that $\ch_\nu(\kappa_I(\nu)) = (\nu-1)^{|I|}  \Pi(\nu)_{\comp(I)}$.

(a) The assertion follows from Proposition~\ref{thm: product formula of superclass identifier} and the equality
\[
\left(\frac{1}{1-\nu}\right)^{|K \cap c_2(A)|}  (\nu-1)^{|K|} = (-1)^{|K \cap c_2(A)|}  (\nu-1)^{|K \setminus c_2(A)|}.
\]

(b) The assertion follows from Proposition~\ref{thm: coproduct formula of superclass identifier}.

\end{proof}

We are now ready to state the main result of this subsection.

\begin{theorem}\label{thm: str consts for B}
Let $m$, $n$, and $k$ be nonnegative integers.

\begin{enumerate}[label = {\rm (\alph*)}] 
    \item Let $K \subseteq [k-1]$. Then
    \[
    \triangle \mathcal{B}(q,t)_{\comp(K)} =\sum_{\substack{m+n = k\\I \subseteq [m-1] \\ J \subseteq [n-1] }} C^K_{I,J}(q,t) (\mathcal{B}(q,t)_{\comp(I)} \otimes \mathcal{B}(q,t)_{\comp(J)}),
    \] where
    \[
    C^K_{I,J}(q,t) = t^{-|I|-|J|} \sum_{\substack{A \in \binom{[m+n]}{n}:\\ (I \#_A J) \cap c(A) = \emptyset \\ I \#_A J \subseteq K \subseteq (I \#_A J ) \cup c(A) }} (q+t)^{|K \cap c_2(A)|}  t^{|K \setminus c_2(A)|}.
    \]
    \item For $\alpha \in \Comp_m,\beta \in \Comp_n$,
\[
\mathcal{B}(q,t)_{\alpha} \mathcal{B}(q,t)_{\beta} = \mathcal{B}(q,t)_{\alpha \odot \beta}.
\]
\end{enumerate}
\end{theorem}

\begin{proof}
(a) For each integer $\nu>1$ we have that $\mathcal{B}(-\nu,\nu-1)_{\comp(I)} = \Pi(\nu)_{\comp(I)}^*$ by Lemma~\ref{lem: transition between Pi and M} (a).
Let 
$$
\triangle \mathcal{B}(-\nu,\nu-1)_{\comp(K)} = \sum_{I,J} C^K_{I,J}(\nu)  (\mathcal{B}(-\nu,\nu-1)_{\comp(I)} \otimes  \mathcal{B}(-\nu,\nu-1)_{\comp(J)}).
$$
By Lemma~\ref{lem: str consts for Pi}, we have 
$$
C^K_{I,J}(\nu) =  \frac{1}{(\nu-1)^{|I|+|J|}} \sum_{\substack{A \in \binom{[m+n]}{n}:\\ (I \#_A J) \cap c(A) = \emptyset \\ I \#_A J \subseteq K \subseteq (I \#_A J ) \cup c(A) }} (-1)^{|K \cap c_2(A)|} (\nu-1)^{|K \setminus c_2(A)|}.
$$
Since $\nu$ ranges over the infinite set $\{2, 3, \ldots \}$,  
$$
C^K_{I,J}(q) =  \frac{1}{(q-1)^{|I|+|J|}} \sum_{\substack{A \in \binom{[m+n]}{n}:\\ (I \#_A J) \cap c(A) = \emptyset \\ I \#_A J \subseteq K \subseteq (I \#_A J ) \cup c(A) }} (-1)^{|K \cap c_2(A)|} (q-1)^{|K \setminus c_2(A)|}
$$
as rational functions in $q$.
On the other hand, one can easily see that 
$$
\mathcal{B}(q,t)_{\comp(K)} = (-q-t)^{|K|} \mathcal{B}(-Q,Q-1)_{\comp(K)}, 
$$
where $Q =\frac{q}{q+t}$.
Taking $\triangle$ on both sides, we obtain the desired result.

(b) The assertions can be obtained by using similar arguments as in (a).
\end{proof}

In the rest of this subsection, we extend Theorem~\ref{thm: str consts for B} 
to $\NSym_{\mathbf{k}(q,t)}$, where $\mathbf{k}$ is an arbitrary nonzero field with identity ${\bf 1}$.
Let us explain the problem in more detail.
Irregardless of the field $\mathbf{k}$, the coproduct of $\NSym_{\mathbf{k}(q,t)}$
is universally expressed by the rule:
$$\triangle H_n= \sum_{i+j=n} H_i \otimes H_j$$
(for more discussion, we refer the readers to~\cite[Thm5.4.2]{GR20}).
However, it is not obvious if this phenomenon happens for other bases, particularly for  
$(\{\mathcal{B}(q,t)_{\alpha}\}$.

For our purpose, we first observe that 
Definition~\ref{defn of B(q,t)} still works for $\NSym_{\mathbf{k}(q,t)}$.
This is because the diagonal entries in the transition matrix from $(\{\mathcal{B}(q,t)_{\alpha}\}, \preceq^c_{{\rm lin.ext.}})$ to $(\{H_{\alpha}\}, \preceq_{{\rm lin.ext.}})$ are all units in $\mathbf{k}(q,t)$.
Next, let us consider the ring homomorphism \[- : \mathbb Z \to \mathbf{k}, \quad a \mapsto \overline {a}:= a \cdot {\bf 1}.\]
Clearly it induces  
the ring homomorphism 
\[-: \mathbb Z[q,t] \to \mathbf{k}[q,t], \quad \sum_{i,j}a_{ij}q^it^j \mapsto \sum_{i,j}\overline {a_{ij}}q^it^j.\]
The following corollary shows that the coproduct rule of $\NSym_{\mathbf{k}(q,t)}$ is universally described not only for $\{H_{\alpha}\}$ but also for $\{\mathcal{B}(q,t)_{\alpha}\}$.
\begin{corollary}
Let $\mathbf{k}$ be an arbitrary nonzero field.
\begin{enumerate}[label = {\rm (\alph*)}] 
    \item 
    The polynomials $C^K_{I,J}(q,t)$ in Theorem~\ref{thm: str consts for B}(a) belong to $\mathbb Z[q,t]$.
\item
Let $k\ge 0$ and $K \subseteq [k-1]$. Then
    \[
    \triangle \mathcal{B}(q,t)_{\comp(K)} =\sum_{\substack{m+n = k\\I \subseteq [m-1] \\ J \subseteq [n-1] }} \overline{C^K_{I,J}(q,t)} (\mathcal{B}(q,t)_{\comp(I)} \otimes \mathcal{B}(q,t)_{\comp(J)}).
    \]
\item For $\alpha \in \Comp_m,\beta \in \Comp_n$,
\[
\mathcal{B}(q,t)_{\alpha} \mathcal{B}(q,t)_{\beta} = \mathcal{B}(q,t)_{\alpha \odot \beta}.
\]

\end{enumerate}
\end{corollary}

\begin{proof}
(a) Note that $|I \#_A J| = |I|+|J|$. 
Using this, one can see that 
$|I|+|J| \le |K \setminus c_2(A)|$ 
if $(I \#_A J) \cap c(A) = \emptyset$ and $I \#_A J \subseteq K$.  

(b) The assertion follows from (a).

(c) The assertion can be obtained in the same way as (b) is obtained from (a).

\end{proof}

\subsection{A variant of $\{\mathcal{B}(q,t)_{\alpha}\mid \alpha \in \Comp\}$}

In this subsection, we consider a variant $\{\mathcal{\widehat{B}}(q,t)_{\alpha}\mid \alpha \in \Comp\}$ of $\{\mathcal{B}(q,t)_{\alpha}\mid \alpha \in \Comp\}$, where 
$\mathcal{\widehat{B}}(q,t)_{\alpha} := \mathcal{B}(q,t)_{\alpha^c}$.
The main reason we are considering it is because
most formulas can be expressed in simpler forms on this basis compared to $\{\mathcal{B}(q,t)_{\alpha}\mid \alpha \in \Comp\}$.

For $k\ge 0$, we set $\mathcal{\widehat{B}}(q,t)_{k} := \mathcal{\widehat{B}}(q,t)_{(k)}$. 
Generally, we call an $\mathbf{k}$-basis $\{ B_{\alpha} \mid \alpha \in \Comp \}$ of $\QSym_\mathbf{k}$ or $\NSym_\mathbf{k}$ \emph{multiplicative} if 
    $B_{\alpha} B_{\beta} = B_{\alpha \cdot \beta}$
    for $\alpha, \beta \in \Comp$.
Similarly, we call an $\mathbf{k}$-basis $\{ B_{\lambda} \mid \lambda \in \Par \}$ of $\Sym_\mathbf{k}$ \emph{multiplicative} if 
    $B_{\lambda} B_{\mu} = B_{\lambda \cdot \mu}$
    for $\lambda, \mu \in \Par$.
And, we denote by $\{\Pi(q)_{\alpha}^* \mid \alpha \in \Comp_n \}$
the dual basis of $\{\Pi(q)_{\alpha} \mid \alpha \in \Comp_n \}$.
In the following, we will investigate how Theorem~\ref{thm: str consts for B} (a) is written for $\mathcal{\widehat{B}}(q,t)_{k}$.
In fact, it can be written in a much simpler form than it is now.
To see this, for $A \subseteq [k]$, we let ${\rm conn}(A) := (A_1, A_2, \ldots, A_l)$ and ${\rm conn}(A^c):= ((A^c)_1, (A^c)_2, \ldots, (A^c)_l)$, i.e., the collections of maximally connected subsets of $A$ and $A^c$, respectively.
Then we define two compositions $\alpha_A$ and $\beta_A$ as follows:
\[
\alpha_A := (|A_1|, |A_2|, \ldots, |A_l|)
\quad \text{and} \quad
\beta_A := (|(A^c)_1|,|(A^c)_2|, \ldots, |(A^c)_l|).
\]

The main result of this subsection is stated in the following form.

\begin{theorem}\label{thm: str consts for hat B}
Let $k$ be a nonnegative integer. Then we have
\begin{enumerate}[label = {\rm (\alph*)}] 
     \item The basis $\{\mathcal{\widehat{B}}(q,t)_{\alpha}\}$ is multiplicative.
    
    \item $\{\mathcal{\widehat{B}}(q,t)_k \mid k \ge 0 \}$ is a generating set of $\NSym_{\mathbb{C}(q,t)}$. 

    \item For the generators $\mathcal{\widehat{B}}(q,t)_k$, the coproduct formula in Theorem~\ref{thm: str consts for B} (a) reads as follows:
        \[
        \triangle \mathcal{\widehat{B}}(q,t)_k = \sum_{A \subseteq [k]} (q+t)^{|c_2(A)|} t^{|c_1(A)|} \, \, \left( \mathcal{\widehat{B}}(q,t)_{\alpha_A} \otimes \mathcal{\widehat{B}}(q,t)_{\beta_A} \right).
        \]
        In particular, by letting $(q,t)=(1,0)$ and $(q,t)=(-1,1)$, we can recover
        the formulas: 
        \[
        \triangle H_k = \sum_{i+j=k} H_i \otimes H_j \quad (\text{in }\NSym), 
        \]
        \[
        \triangle \Lambda_k = \sum_{i+j=k} \Lambda_i \otimes \Lambda_j \quad (\text{in }\NSym). 
        \]
\end{enumerate}
\end{theorem}

\begin{proof}
(a) Let $\widehat{\Pi}(q)_{\alpha}^* :=\Pi(q)_{\alpha^c}^*$.
One can show
\[
(\alpha \odot \beta)^c = \alpha^c \cdot \beta^c,
\]
which follows from~\eqref{eq: concatenation, near concatenation in set}.
Thus, by Lemma~\ref{lem: str consts for Pi} (b), $\{\widehat{\Pi}(q)_{\alpha}^*: \alpha \in \Comp_n\}$ is multiplicative.
Now, the assertions can be obtained by using similar arguments as in the proof of Theorem~\ref{thm: str consts for B}.

(b) By (a), we have 
\[\mathcal{\widehat{B}}(q,t)_{\alpha}=\mathcal{\widehat{B}}(q,t)_{\alpha_1}\mathcal{\widehat{B}}(q,t)_{\alpha_2}\cdots \mathcal{\widehat{B}}(q,t)_{\alpha_{l}}\]
for every composition $\alpha=(\alpha_1, \ldots, \alpha_l)$,  
which proves the assertion.

(c) Fix an integer $0\le m \le k$ and let $n:=k-m$ and $K:=[k-1]$.
For $I \subseteq [m-1]$, and $J \subseteq [n-1]$, 
we showed in Theorem~\ref{thm: str consts for B} (a) that 
\[
C^K_{I,J}(q,t) = \sum_{\substack{A \in \binom{[m+n]}{n}:\\ (I \#_A J) \cap c(A) = \emptyset \\ I \#_A J \subseteq K \subseteq (I \#_A J ) \cup c(A) }} (q+t)^{|c_2(A)|}  t^{(k-1)-|I| -|J|-|c_2(A)|}.
\]
Choose a subset $A \in \binom{[m+n]}{n}$.
Let $(I_A, J_A)$ be a pair satisfying the conditions:
\begin{itemize}
    \item $I_A \#_A J_A \subseteq [m+n-1] \setminus c(A)$, and 
    \item $I_A \#_A J_A \subseteq K \subseteq (I_A \#_A J_A ) \cup c(A).$
\end{itemize}
Then
\[(I_A \#_A J_A) \sqcup c(A) = [m+n-1],\]
which implies that $(I_A, J_A)$ is uniquely determined by $A$. To be precise,
\begin{align*}
&I_A = [m-1] \setminus \{|(A^c)_1|, |(A^c)_1|+|(A^c)_2|, \ldots, (|(A^c)_1|+|(A^c)_2|+ \ldots +|(A^c)_{l-1}|) \}, 
\\
&J_A = [n-1] \setminus \{ |A_1|, |A_1|+|A_2|, \ldots, , (|A_1|+|A_2|+ \ldots +|A_{l-1}|) \}.
\end{align*}
Therefore, 
\[
C^K_{I,J}(q,t) = 
\begin{cases}
(q+t)^{|c_2(A)|}  t^{(k-1)-|I| -|J|-|c_2(A)|} &\text{ if } I = I_A, J = J_A \text{ for some } A,\\
0 &\text{ otherwise.}
\end{cases}
\]
Note that 
$|I_A| = (m-1) - |\rm conn(A^c)| -1$, $|J_A| = (n-1) - |\rm conn(A)| -1$, and
\begin{align*}
&|{\rm conn(A)}| =
\begin{cases}
|c_1(A)| &\text{ if } k \in A_l,\\
|c_1(A)|+1 &\text{ if } k \notin A_l
\end{cases}\\
&|{\rm conn(A^c)}| =
\begin{cases}
|c_2(A)| &\text{ if } k \in (A^c)_l,\\
|c_2(A)|+1 &\text{ if } k \notin (A^c)_l.
\end{cases}
\end{align*}
Furthermore, we see that $k \in (A^c)_l$ if and only if $k \notin A_l$, thus 
\[
(k-1)-|I_A| -|J_A|-|c_2(A)| =
|{\rm conn(A)}| + |{\rm conn(A^c)}| -|c_2(A)| -1 = |c_1(A)|.
\]
Now the first assertion follows from the observation that $ \alpha_A= \comp((I_A)^c)$ and $\beta_A=\comp((J_A)^c)$.

Since 
\[
|c_1(A)| = 0
\quad \text{if and only if} \quad 
A = [k] \setminus [i] \text{ for some } 0 \le i \le k,
\]
\[
|c_2(A)| = 0
\quad \text{if and only if} \quad 
A^c = [k] \setminus [i] \text{ for some } 0 \le i \le k,
\]
we have
\begin{align*}
    C^K_{I,J}(1,0) = 
    \begin{cases}
    1 &\text{ if } I = \{1,2,\ldots,i-1 \}, J = \{i+1,i+2,\ldots,k-1 \} \text{ for some } 0 \le i \le k,\\
    0 &\text{ otherwise,}
    \end{cases}
\end{align*}
\begin{align*}
    C^K_{I,J}(-1,1) = 
    \begin{cases}
    1 &\text{ if } J = \{1,2,\ldots,i-1 \}, I = \{i+1,i+2,\ldots,k-1 \} \text{ for some } 0 \le i \le k,\\
    0 &\text{ otherwise.}
    \end{cases}
\end{align*}
This justifies the second assertion.
\end{proof}

\begin{example}
The following table shows the calculations of $c_1(A), c_2(A), \alpha_A,$ and $\beta_A$ for each $A \in [3]$.

%column 너비 조정, centering
\newcolumntype{C}[1]{>{\centering\arraybackslash}p{#1}}

%horizontal 굵은 줄
\makeatletter
\def\hlinewd#1{%
\noalign{\ifnum0=`}\fi\hrule \@height #1 \futurelet
\reserved@a\@xhline}
\makeatother

%vertical 굵은 줄
\newcolumntype{?}{!{\vrule width 1pt}}

%font크기 조절
%\tiny, \scriptsize, \footnotesize, \small, \normalsize, \large, \Large, \LARGE, \huge, and \Huge.

%table 시작
\begin{table}[ht]
\small{
\begin{tabular}
%{|c?c|c|c|c|}
{|C{1.5cm}?C{1.5cm}|C{1.5cm}|C{1.5cm}|C{1.5cm}|}
%{|p{1.8cm}?p{2.15cm}|p{2.15cm}|p{2.15cm}|p{2.15cm}|}
\hline
& & & &  \\[-1em]
 \textbf{$A$}  & \textbf{$c_1(A)$}       & \textbf{$c_2(A)$}        &  \textbf{$\alpha_A$}  &  \textbf{$\beta_A$}   \\[-1em]
& & & &  \\ \hlinewd{1pt}
 $\emptyset$   & $\emptyset$                    & $\emptyset$  & $\emptyset$           &  (3)                      \\ \hline
\{1\}    & \{1\}                 & $\emptyset$                    & (1)       & (2)            \\ \hline
\{2\}    & \{2\}                    & \{1\}                    & (1)         & (1,1)          \\ \hline
\{3\}    & $\emptyset$                  & \{2\}                    & (1)       & (2)          \\ \hline
\{1,2\}    & \{2\}                  & $\emptyset$                    & (2)       & (1)           \\ \hline
\{1,3\}    & \{1\}                    & \{2\}                    & (1,1)        & (1)           \\ \hline
\{2,3\}    & $\emptyset$                    & \{1\}                    & (2)         & (1)          \\ \hline
\{1,2,3\}    & $\emptyset$                    & $\emptyset$                  & (3)       & $\emptyset$               \\ \hline
\end{tabular}}
\end{table}

\noindent
Thus, Theorem~\ref{thm: str consts for hat B} (c) says that 
\begin{align*}
\triangle \mathcal{\widehat{B}}(q,t)_3 &= 
 \left( 1 \otimes \mathcal{\widehat{B}}(q,t)_3 \right)
 + (q+2t) \left( \mathcal{\widehat{B}}(q,t)_1 \otimes \mathcal{\widehat{B}}(q,t)_2 \right)
+ (q+t) t \left( \mathcal{\widehat{B}}(q,t)_1 \otimes \mathcal{\widehat{B}}(q,t)_{(1,1)} \right)\\
 &\, \quad + t \left( \mathcal{\widehat{B}}(q,t)_2 \otimes \mathcal{\widehat{B}}(q,t)_1 \right)
 + (q+t) t \left( \mathcal{\widehat{B}}(q,t)_{(1,1)} \otimes \mathcal{\widehat{B}}(q,t)_1 \right)\\
 &\, \quad + (q+t)  \left( \mathcal{\widehat{B}}(q,t)_2 \otimes \mathcal{\widehat{B}}(q,t)_1 \right)
+ \left( \mathcal{\widehat{B}}(q,t)_3 \otimes 1 \right).
\end{align*}
\end{example}

In the rest of this subsection, 
we provide some interesting consequences obtained by specializations of $q$ and $t$. 
The first consequence can be obtained by combining Theorem~\ref{thm: str consts for B} and the surjective Hopf algebra homomorphism in~\eqref{def of commutative image}
\[{\rm comm}:\NSym \to \Sym, \quad H_n \mapsto h_n.\]

\begin{corollary}\label{minimal geberating sets for nsym and sym}
Let $a \in \mathbb{C}\setminus\{0\}$ and  $b \in \mathbb{C}$. Then the following hold.
\begin{enumerate}[label = {\rm (\alph*)}]
\item
$ \{ \mathcal{\widehat{B}}(a,b)_{n} \mid n \ge 0 \}$ is a generating set of $\NSym$.
\item
$ \{ {\rm comm}(\mathcal{\widehat{B}}(a,b)_{n}) \mid n \ge 0 \}$ is a generating set of $\Sym$. 
\end{enumerate}
\end{corollary}

\begin{proof}
(a) The diagonal entries of the transition matrix from $\mathcal{\widehat{B}}(a,b)_{\comp(I)}$ to $H_{\comp(J)}$ are non-zero if $a \neq 0$, and therefore 
$ \{ \mathcal{\widehat{B}}(a,b)_{\alpha} \mid \alpha \in \Comp \}$ is a basis for $\NSym$ if $a \neq 0$.
On the other hand, from Theorem~\ref{thm: str consts for hat B} (a) it follows that all $\mathcal{\widehat{B}}(a,b)_{\alpha}$'s are 
generated by $\mathcal{\widehat{B}}(a,b)_{n}\, (n\ge 0)$.

(b) Since ${\rm comm}$ is an algebra homomorphism, the assertion follows from (a).
\end{proof}

It should be remarked that $\mathcal{\widehat{B}}(a,b)_{n}$ and 
${\rm comm}(\mathcal{\widehat{B}}(a,b)_{n})$ in Corollary~\ref{minimal geberating sets for nsym and sym} can be written in a more concrete form as follows:
\begin{align*}
    \mathcal{\widehat{B}}(a,b)_{n} = \sum_{J \subseteq [n-1]} a^{(n-1) - |J|} b^{|J|} H_{\comp(J)}
    = \sum_{\beta \in \Comp_n} a^{n-\ell(\beta)} b^{\ell(\beta)-1} H_{\beta},
\end{align*}
and therefore
\begin{align*}
    {\rm comm}(\mathcal{\widehat{B}}(a,b)_{n}) = \sum_{\beta \in \Comp_n} a^{n-\ell(\beta)} b^{\ell(\beta)-1} h_{\lambda(\beta)}= \sum_{\lambda \in \Par_n} a^{n-\ell(\lambda)} b^{\ell(\lambda)-1} C_{\lambda} \, \, h_{\lambda},
\end{align*}
where $C_{\lambda} = {\ell(\lambda)!}/{\prod_{i} m_i(\lambda)!}$ for $\lambda = 1^{m_1(\lambda)} 2^{m_2(\lambda)} \cdots $.

The second consequence concerns a symmetry on the set $ \{ \mathcal{\widehat{B}}(a,b)_{\alpha} \mid a, b \in \mathbb{C}, \alpha \in \Comp  \}$.
Recall that $\NSym$ has the involutive anti-homomorphism 
\[\omega:\NSym \to \NSym, \quad H_{\alpha} \mapsto \Lambda_{\alpha^r} 
\,\,\,(\alpha \in \Comp)\]
(see~\cite[Chapter 4]{Gelfand95}). 
Proposition~\ref{prop: specializations of B(q,t)} (a) and (b) tell us that $\omega(\mathcal{\widehat{B}}(1,0)_{\alpha}) = \mathcal{\widehat{B}}(-1,1)_{\alpha^r}$.
In fact, this is a special case of the following proposition.

\begin{proposition}\label{prop: omega involution}
For any $\alpha \in \Comp$
and $a,b \in \mathbb{C}$, 
\[
\omega(\mathcal{\widehat{B}}(a,b)_{\alpha}) = \mathcal{\widehat{B}}(-a,a+b)_{\alpha^r}.
\]
\end{proposition}
\begin{proof}
It is straightforward to check the equality when $a=0$, so we assume $a \neq 0$.
Since $\omega$ is an involutive anti-homomorphism and $\{\mathcal{\widehat{B}}(a,b)_{\alpha} \}$ is multiplicative by Theorem~\ref{thm: str consts for hat B} (a), 
we only have to show that $\omega(\mathcal{B}(-a,a+b)_{\comp([n-1])}) = \mathcal{B}(a,b)_{\comp([n-1])}$ for each $n \ge 1$.
Note that we have
\begin{align*}
    \omega(\mathcal{B}(-a,a+b)_{\comp([n-1])}) &= \sum_{J \subseteq [n-1]} (-a)^{(n-1)-|J|} (a+b)^{|J|} \Lambda_{\comp(J)^r}\\
    &= \sum_{J \subseteq [n-1]} (-a)^{(n-1)-|J|} (a+b)^{|J|} \Lambda_{\comp(J)},
\end{align*}
since $|\set(\comp(J)^r)| = |J|$.
By~\cite[Proposition 4.3]{Gelfand95}, we have
\[
\Lambda_{\comp(J)} = \sum_{J \subseteq K} (-1)^{(n-1) - |K|} H_{\comp(K)}
\]
for $J \subseteq [n-1]$. Therefore,
\begin{align*}
    \omega(\mathcal{B}(-a,a+b)_{\comp([n-1])}) &= \sum_{J \subseteq [n-1]} (-a)^{(n-1)-|J|} (a+b)^{|J|} \left( \sum_{K: J \subseteq K} (-1)^{(n-1) - |K|} H_{\comp(K)} \right)\\
    &= \sum_{K \subseteq [n-1]} (-1)^{(n-1) - |K|} \left( \sum_{J: J \subseteq K} (-a)^{(n-1)-|J|} (a+b)^{|J|} \right) H_{\comp(K)}\\
    &=\sum_{K \subseteq [n-1]} (-1)^{(n-1) - |K|} (-a)^{n-1} \left( \dfrac{a+b}{-a} +1  \right)^{|K|}  H_{\comp(K)}\\
    &=\sum_{K \subseteq [n-1]} a^{(n-1) - |K|} \, b^{|K|} \, H_{\comp(K)}\\
    &=\mathcal{B}(a,b)_{\comp([n-1])}.
\end{align*}
\end{proof}

Let $\Comp_{n,l}$ (resp. $\Par_{n,l}$) be the set of compositions (resp. partitions) of $n$ with length $l$.
Then $\SG_l$ acts on $\Comp_{n,l}$ on the right by place permutations, i.e.,
\[(\alpha_1, \alpha_2, \ldots, \alpha_l)\cdot \sigma= (\alpha_{\sigma(1)}, \alpha_{\sigma(2)}, \ldots, \alpha_{\sigma(l)}).\]
Clearly $\Par_{n,l}$ is a complete set of orbit-representatives.
With this preparation, we state the third consequence.

\begin{proposition}\label{prop: multiplicative basis for Sym}
Let $a \in \mathbb{C} \setminus \{0\}$ and  $b \in \mathbb{C}$.
Then $\{ {\rm comm}(\mathcal{\widehat{B}}(a,b)_{\lambda}) \mid \lambda \in \Par \}$ is a multiplicative basis for $\Sym$. 
\end{proposition}

\begin{proof}
Since $\{ \mathcal{\widehat{B}}(a,b)_{\alpha}) \mid \alpha \in \Comp_n\}$ is a multiplicative basis for $\NSym$ by Theorem~\ref{thm: str consts for hat B} (a) and ${\rm comm} : \NSym \to \Sym$ is a surjective algebra homomorphism, we have only to show that 
$\{ {\rm comm}(\mathcal{\widehat{B}}(a,b)_{\lambda}) \mid \lambda \in \Par \}$ is a basis for $\Sym$.
Note that $\{ {\rm comm}(\mathcal{\widehat{B}}(a,b)_{\alpha}) \mid \alpha \in \Comp_n\}$ is a spanning set of $\Sym_n$ for each $n\ge 0$. 
Hence, for our purpose, it suffices to show that 
${\rm comm}(\mathcal{\widehat{B}}(a,b)_{\alpha})
={\rm comm}(\mathcal{\widehat{B}}(a,b)_{\beta})$ 
whenever $\alpha, \beta \in \Comp_{n,l}$ are in the same orbit. 
To see this, we first recall that 
\[
\mathcal{\widehat{B}}(q,t)_{\alpha} = \sum_{\gamma \preceq \alpha} q^{|\set(\alpha)^c \setminus \set(\gamma)|} \  t^{|\set(\alpha)^c \cap \set(\gamma)|} \  H_{\gamma}.
\]
Let 
$\Gamma_{\alpha} := \{\gamma \mid \gamma \preceq \alpha \}$ and
$\Gamma_{\beta} := \{\gamma \mid \gamma \preceq \beta \}$.
Since $ {\rm comm}(H_{\alpha}) =  {\rm comm}(H_{\beta})=h_{\lambda(\alpha)}$, 
we have only to see that there is a bijection $\Phi: \Gamma_{\alpha} \to \Gamma_{\beta}$ such that
$\lambda(\gamma) = \lambda(\Phi(\gamma))$ and $|\set(\alpha)^c \cap \set(\gamma)| = |\set(\beta)^c \cap \set(\Phi(\gamma))|$.
For $\gamma \preceq \alpha$, we write it as 
$$
(\gamma^{(1)};\gamma^{(2)}; \ldots ; \gamma^{(l)}) \text{ where } \gamma^{(i)} \in \Comp_{\alpha_i}.  
$$
Then we define
$\Phi : \Gamma_{\alpha} \ra \Gamma_{\beta} $
by
\[
\Phi(\gamma) = (\gamma^{(\sigma(1))};\gamma^{(\sigma(2))}; \ldots ; \gamma^{(\sigma(l))}).
\]
It can be easily seen that $\Phi$ satisfies the desired properties, so we are done.
\end{proof}

The final consequence to be introduced 
concerns the structure constants of $\QSym$ for the 
basis $\{M_\alpha\}$ of the monomial quasisymmetric functions.
It is well known that 
\[M_{\alpha} M_{\beta} = \sum_{\gamma} e^{\gamma}_{\alpha, \beta} M_{\gamma}, \]
where 
$e^{\gamma}_{\alpha, \beta}$ counts the {\it overlapping shuffles} of $\alpha$ and $\beta$ with weight $\gamma$.
For more information, we refer the readers to~\cite[Proposition 5.1.3]{GR20}. 
Putting Proposition~\ref{prop: specializations of B(q,t)} and Theorem~\ref{thm: str consts for B} together, one can derive the following combinatorial descriptions of the overlapping shuffles.

\begin{corollary}\label{cor: overlapping shuffle}
Let $I \subseteq [m-1], J \subseteq [n-1]$, and $K \subseteq [m+n-1]$.
The following three sets have the same cardinality.
\begin{enumerate}[label = {\rm (\alph*)}]
    \item The set of overlapping shuffles of $\comp(I^c)$ and $\comp(J^c)$ with weight $\comp(K^c)$
    
    \item $$\left\{ A \in \binom{[m+n]}{n}\ \mid \begin{array}{l}
    (I \#_A J) \cap c(A) = \emptyset, \\
    I \#_A J \subseteq K \subseteq (I \#_A J ) \cup c(A),\\
    |K \setminus c_2(A)| = |I| + |J|
   \end{array}\right\}$$
    
    \item $$ \left\{ A \in \binom{[m+n]}{n}\ \mid \begin{array}{l}
    (I \#_A J) \cap c(A) = \emptyset, \\
    I \#_A J \subseteq K \subseteq (I \#_A J ) \cup c(A),\\
    K \cap c_2(A) = \emptyset
  \end{array}\right\}$$
\end{enumerate}
\end{corollary}

\begin{proof}
Proposition~\ref{prop: specializations of B(q,t)} (a) says that 
$\mathcal{B}(1,0)_{\comp(K)} = H_{\comp(K^c)}$. 
Therefore, from the fact that $\{H_{\alpha}\}$ is the dual basis of $\{M_{\alpha}\}$
it follows that the cardinality of $(a)$ is given by 
$C^K_{I,J}(1,0)$ (=the cardinality of (b)).

On the other hand, Proposition~\ref{prop: specializations of B(q,t)} (a) says that $\mathcal{B}(-1,1)_{\comp(K)} = \Lambda_{\comp(K^c)}$. 
It was shown in~\cite[Proposition 3.8]{Gelfand95} that $\Lambda_{\alpha}\Lambda_{\beta}$ and $H_{\alpha}H_{\beta}$ have the same structure constants, thus 
$C^K_{I,J}(1,0)$ equals $C^K_{I,J}(-1,1)$(=the cardinality of (c)).
\end{proof}

\bibliographystyle{plain}
\bibliography{references}
\end{document}